\documentclass[reqno, 11pt,letter]{amsart}
\allowdisplaybreaks
\usepackage{graphicx,  enumitem, amsmath, amssymb,amsthm,hyperref,amscd,xcolor}
\usepackage{mathabx,manfnt,braket,mathrsfs, comment,mathtools,bm}
\usepackage{tikz}
\usepackage{cancel}
\usetikzlibrary{arrows, matrix,decorations.pathmorphing}
\usepackage[utf8]{inputenc}


\newtheorem{thm}{Theorem}[section]
\newtheorem{prop}{Proposition}[section]

\newtheorem{lem}[prop]{Lemma}
\newtheorem{cor}[prop]{Corollary}
\theoremstyle{definition}

\newtheorem*{rem}{Remark}

\numberwithin{equation}{section}

\newcommand{\rot}{\mathrm{R}}
\newcommand{\tra}{\mathrm{M}}
\newcommand{\all}[1]{\left(#1\right)}
\newcommand{\norm}[1]{\left\Vert#1\right\Vert}
\newcommand{\abs}[1]{\left\vert#1\right\vert}
\newcommand{\cx}{{\mathbb{C}}}
\newcommand{\C}{{\mathbb{C}}}
\newcommand{\rl}{{\mathbb{R}}}
\newcommand{\T}{\mathbb{T}}
\newcommand{\D}{\mathbb{D}}
\newcommand{\dd}{\bm{D}}
\newcommand{\db}{\bm{\overline{D}}}
\newcommand{\DD}{\mathscr{D}}
\newcommand{\OO}{\mathscr{O}}

\newcommand{\Sl}{\sigma_{\lambda}}

\newcommand{\Slalpha}{\sigma_{\lambda_{\alpha}}}
\newcommand{\ppi}{\boldsymbol{\pi}}
\newcommand{\pialphasigma}{\ppi_{\alpha}^{\sigma}}
\newcommand{\pialpha}{\ppi_{\alpha}}
\newcommand{\om}{\mathscr{O}_M}

\newcommand{\N}{\mathbb{N}}
\newcommand{\Z}{\mathbb{Z}}

\newcommand{\ol}{\overline}

\newcommand{\dbar}{\ol\partial}

\newcommand{\wt}{\widetilde}
\newcommand{\wh}{\widehat}

\newcommand\ipr[1]{\left\langle #1 \right\rangle}

\title{Power series as Fourier Series}
\author{Debraj Chakrabarti}
\address{Department of Mathematics, Central Michigan University, Mount Pleasant,  MI 48859,  USA}
\email{chakr2d@cmich.edu}

\author{Anirban Dawn}
\address{Department of Mathematics, The University of Tampa, Tampa,  FL 33606,  USA}
\email{adawn@ut.edu}

\thanks{Debraj Chakrabarti was partially supported by  Simons Foundation Collaboration Grant number 706445.}
\begin{document}
	\begin{abstract} An abstract theory of Fourier series in locally convex topological vector spaces is developed.  An analog of  Fej\'{e}r's theorem is proved for these series. The theory 
		is applied to distributional solutions of Cauchy-Riemann equations to recover basic results of complex analysis. Some classical results of function theory are also shown to be consequences
		of the series expansion. 
\end{abstract}

\maketitle

\section{Introduction}

\subsection{Motivation} Two types of series expansion ubiquitous in mathematics  are the power series  of an  analytic function, and the Fourier series of an integrable function on 
the circle. In the complex domain, well-known theorems of 
elementary complex analysis guarantee  the existence of locally uniformly convergent
power series expansions of holomorphic functions in domains with ample symmetry such as disks and annuli,
where ``holomorphic"  can be taken in the sense of Goursat,  i.e.,  the function is complex-differentiable at each point. On the other hand, the 
convergence of a  Fourier series is a subtle matter, and its study has led to many developments 
in  analysis (see \cite{zygmund}). The close connection between these two types of outwardly different series expansions is a recurring theme in many areas of 
classical analysis, e.g., in the theory of Hardy spaces. 

It is however not difficult to see  that  at a certain level of abstraction, Fourier series and power series are in fact two examples of the same 
phenomenon,   the representation theory of the circle group. The aim of this  article is to take this idea seriously, and use it  to recapture some basic results of complex analysis. The
take-home message is that many properties of holomorphic functions, such as the almost-supernatural regularity phenomena are profitably thought of as expressions of 
symmetry, more precisely the invariance of certain locally convex spaces under the Reinhardt action of the torus group on $\cx^n$.
It is hoped that the  point of view taken here has pedagogical as well as conceptual value, and will be of interest to students  of complex analysis.

\subsection{Abstract Fourier series} 

We begin in Section~\ref{sec-abstract} with an account of Fourier series associated to a continuous representation of the
$n$-dimensional torus $\T^n$ on  locally convex topological vector spaces, using some ideas of   \cite{johnson}.
This provides the  unifying language of sufficient generality  to encompass both classical Fourier expansions and the power series representations of complex analysis. Even at this very general and ``soft" level,
one can establish a version of 
 Fej\'{e}r's theorem on the summability in the Ces\`{a}ro sense  (Theorem ~\ref{thm-fejer}), which can be thought of as a ``completeness" statement for the holomorphic monomials.

\subsection{ Analyticity of holomorphic distributions} Using  the abstract framework of Section~\ref{sec-abstract}, we recapture in Section~\ref{sec-distn}, in a novel way,
some of  the basic classical facts about holomorphic functions.
We show that  on a Reinhardt domain,  a distribution which satisfies the Cauchy-Riemann equations  (which we call a \emph{holomorphic distribution})  has a complex power series representation that converges uniformly along with 
derivatives of all orders
on compact subsets of  a ``relatively complete" log-convex Reinhardt domain (see Theorem~\ref{thm-distn} below), and thus a holomorphic distribution
 is a function, in fact an analytic function.   Traditionally, to prove 
such an assertion, one would start by showing that a  holomorphic distribution  is actually a (smooth) function. This can be done either by ad hoc arguments for the Laplacian going back to Weyl (see \cite{weyl1940}), or in a more general way, by noticing  that fundamental solution $\frac{1}{\pi z}$ of the Cauchy-Riemann operator $\frac{\partial}{\partial \ol{z}}$
 is smooth, in fact 
real-analytic, away from the singularity at the origin.  It then follows by standard arguments about convolutions (see \cite[Theorem 4.4.3]{horvol1}), that   any distributional solution of $\frac{\partial}{\partial \ol{z}} u=0$ is real-analytic. 
Therefore, a holomorphic  distribution is  a 
holomorphic function in the sense of Goursat, and  classical results of elementary complex analysis give the power series expansion 
via the Cauchy integral formula. The extension of the domain of convergence to the envelope of holomorphy 
 can be obtained by convexity arguments
(see \cite{range}).

While this classical argument  has the admirable advantage of placing holomorphic functions in the context of solutions of hypoelliptic 
equations, 
it also has the shortcoming that  the crucial property of \emph{complex-analyticity} (and the associated Hartogs phenomenon in several variables) of holomorphic distributions
is proved not from an analysis of the action of a differential operator on distributions, but by falling back on the Cauchy integral formula.
After using the real-analytic hypoellipticity of $\frac{\partial}{\partial \ol{z}}$ to conclude that holomorphic distributions are real-analytic,
we discard all of this information, except that holomorphic functions are $\mathcal{C}^1$ and satisfy the Cauchy-Riemann equations in the classical sense.
 In our approach here, however, complex analyticity 
of holomorphic distributions is proved directly in a conceptually straightforward way, by expanding a holomorphic distribution on a Reinhardt domain 
in a Fourier series, and then showing that the resulting series (the Laurent series of the holomorphic function)  converges  in the
$\mathcal{C}^\infty$-topology, not only on the original set where the distribution was defined, but possibly on a larger domain,
thus underlining the fact that Hartogs phenomenon can be thought of as a regularity property of the solutions of the same nature 
as smoothness. Our proof also clearly locates the origin of the remarkable regularity of  holomorphic distributions in
\begin{itemize}
    \item[(a)] the invariance properties of the space of holomorphic distributions under Reinhardt rotations and translations,
    \item[(b)] symmetry and convexity 
properties of the Laurent monomial functions $z\mapsto z_1^{\alpha_1}\dots z_n^{\alpha_n},$\\
$ \alpha_j\in \Z$ and
\item[(c)] the fact that radially symmetric
holomorphic distributions are constants.
\end{itemize}
It is also interesting that  the fact that $\frac{1}{\pi z}$ is 
a fundamental solution of the Cauchy-Riemann operator, which is  key to many results of complex analysis including the Cauchy integral formula,
does not play any role in  our approach.  

The method of proving analyticity via Fourier expansion can be used in other contexts. For example, replacing the  representation theory of $\T^n$ by that of the special orthogonal group $SO(n)$,
the method can be used to show that a harmonic distribution in a ball of $\rl^n$ is  in fact a real analytic function and 
admits an expansion in solid harmonics (see, e.g., \cite[pp. 316-317]{cour-hilb}), which converges uniformly along with all derivatives on compact subsets of the ball.  Similarly, one 
can obtain the  Taylor/Laurent expansion of a monogenic function of a Clifford-algebra variable in  ``spherical monogenics", the analogs for functions of a Clifford-algebra variable of 
the monomial functions $z\mapsto z^n, n\in \Z$ (see \cite{brakx}).

\subsection{Analyticity of continuous holomorphic functions} While the theory of generalized functions forms the natural context for solution of linear partial differential equations such as the Cauchy-Riemann equations,
for aesthetic and pedagogical reasons   it is natural to ask whether it is possible to 
develop the Fourier approach to power series, as outlined in Section~\ref{sec-distn}, 
using only classical notions of functions as continuous mappings, and  derivatives as 
limits of difference quotients, and without anachronistically invoking distributions.  We accomplish this in the last Section~\ref{sec-morera} of this paper. We start from 
the assumption that a continuous function on the complex complex plane satisfies the hypothesis of the Morera theorem, and show directly that it is complex-analytic \emph{without any recourse to 
the Cauchy integral representation formula} (which, being a case of Stoke's theorem,  requires the differentiability of the function, at least at each point). Not unexpectedly, one of the steps of the proof uses the classical triangle-division used 
in the standard proof of the Cauchy theorem for triangles.

\subsection{Acknowledgements} The first author would like to thank Luke Edholm and Jeff McNeal for many interesting conversations about the topic of power series. He would also like 
to thank the students of MTH 636 and MTH 637 at Central Michigan University  over the years for their many questions, which led him to think about the true significance of power series expansions.

Sections~\ref{sec-abstract} and \ref{sec-missing} of this paper are based on part of the Ph.D. thesis of the second author under the supervision of the first author. The second author would also like to thank his other committee members, Dmitry Zakharov and Sonmez Sahutoglu for their support. Other results from the thesis have appeared in the paper \cite{dawn}. 
\section{Fourier Series in Locally Convex spaces}\label{sec-abstract}
\subsection{Nets, series  and integrals in LCTVS} We begin by recalling some notions and facts about functional analysis in topological vector spaces. See the textbooks \cite{treves,rudin,bourbaki} for more details on these matters.

 Let $X$ be a locally convex Hausdorff topological vector space (we use the standard abbreviation \emph{LCTVS}). Recall that the topology of  $X$ can  also be described by prescribing the family of \emph{continuous seminorms} on $X$:
a net $\{x_j\}$ in $X$ converges to $x$, if and only if for each continuous seminorm $p$ on $X$, we have $p(x-x_j)\to 0$ as a net of real numbers.
In practice, we describe the topology of an LCTVS by  specifying a \emph{generating family} of seminorms (analogous to describing a topology by a subbasis):  a collection of continuous seminorms $\{p_k : k \in K\}$ on $X$ is said to \emph{generate the topology} of $X$ if for every continuous seminorm $q$ on $X$, there exists a finite subset $\{k_1,\dots, k_n\} \subset K$ and a $C > 0$ such that 
\begin{equation}\label{eq:1}
q(x) \leq C \cdot \max\{p_{k_1}(x), \dots, p_{k_n}(x)\}\quad \textrm{for all $x \in X$},
\end{equation}
and further, for every nonzero $x \in X$ there exists at least one $k\in K$ such that $p_k(x) \neq 0$ (this \emph{separating property} ensures that the topology of $X$ is Hausdorff).  Then clearly a net $\{x_j\}$ converges in $X$ if and only if  $p_k(x_{j} - x) \rightarrow 0$
for each $k\in K$.

Let $\Gamma$ be a directed set with order $\geq$. Recall that a net $\{x_{\alpha}\}_{\alpha \in \Gamma}$  in $X$ is said to be \emph{Cauchy } if for every $\epsilon > 0$ and every continuous seminorm $p$ on $X$, there exists $\gamma \in \Gamma$ such that whenever $\alpha, \beta \in \Gamma$ and $\alpha, \beta \geq \gamma$, we have  $p(x_{\alpha} - x_{\beta}) < \epsilon$.  The space $X$ is said to be \emph{complete} if every Cauchy net in $X$ converges.
Observe that  in the above definition we can use a generating family of seminorms rather than all continuous seminorms.

If $\{S_k\}_{k\in \N}$ is a sequence in a vector space, we can define the sequence of the corresponding \emph{Ces\`{a}ro means}
 $\{C_n\}_{n \in \N}$ by  $$	\displaystyle{	C_n= \frac{1}{n+1} \sum \limits_{k=0}^{n} S_k  }. $$The following is the analog for sequences 
 in LCTVS of an elementary fact well-known for sequences of numbers:
\begin{prop}\label{prop-cesaro}
	Let $\{S_k\}_{k \in \N}$ be a convergent sequence in an LCTVS $X$. Then  the sequence of Ces\`{a}ro means $\{C_n\}_{n \in \N}$ is 
	also convergent, and has the same limit.
\end{prop}
\begin{proof} Let $S= \lim\limits_{k\to\infty}S_k$.  If $p$ is a continuous seminorm on $X$ and $\epsilon > 0$, there is $N_1$ such that $p(S_k - S) < \epsilon/2$ for $k \geq N_1$. Set $s = \sum_{k=0}^{N_1} p(S_k - S)$. Then for $n> N_1$, we have
	
\begin{align*}
	p(C_n-S)= p\left(\frac{1}{n+1}\sum \limits_{k=0}^{n} S_k  - S\right)\leq \frac{1}{n+1}\sum\limits_{k=0}^np(S_k-S)< \frac{s}{n+1}+ \frac{n-N_1}{n+1}\cdot\frac{\epsilon}{2}.
\end{align*}
	Therefore, if we choose $N>N_1$ so large that $\frac{s}{n+1}< \frac{\epsilon}{2}$, then for $n\geq N$ we have $p(C_n-S)<\epsilon$, 
	so $ C_n\to S$ in $X$.
\end{proof}

For a  formal series $\displaystyle{\sum_{j=0}^{\infty} x_j}$ in an LCTVS $X$, convergence is defined in the usual way, i.e. the sequence of partial sums converges in $X$.   A formal sum 
${\sum_{\alpha \in \mathfrak{A}} x_{\alpha}}$ over a countable index set $\mathfrak{A}$, where $x_\alpha$ are vectors in an LCTVS $X$,
is said to be \emph{absolutely convergent} if there exists a bijection
$
\tau: \N \rightarrow \mathfrak{A}  
$
such that for every continuous seminorm $p$ on $X$, the series of non-negative real numbers
$\displaystyle{
\sum_{j=0}^{\infty} p (x_{\tau(j)})}
$
is  convergent (see  \cite{kadets1997}).  To check that a series is absolutely convergent, we only need to check the convergence 
of  the above series for seminorms in a fixed generating family. If $X$ is a locally compact Hausdorff space, absolute convergence in the Fréchet space $\mathcal{C}(X)$ 
of  continuous complex valued functions on $X$ is what is classically called \emph{ normal convergence} (see \cite[pp. 104 ff.]{remmert}). Absolute
convergence is typical for many spaces of holomorphic functions, e.g. in the space of holomorphic functions on a Reinhardt domain smooth up to the boundary (see \cite{dawn}).

  The following result, whose proof mimics the corresponding result for numbers,  shows that absolutely convergent series in 
LCTVS behave very much like absolutely convergent series of numbers:
\begin{prop} \label{prop-absolute}Let $X$ be a \emph{complete} LCTVS, and let  ${\sum_{\alpha \in \mathfrak{A}} x_{\alpha}}$  be an absolutely convergent series of elements of $X$.
	Then  the series is \emph{unconditionally convergent}: there is an $s\in X$ such that  for \emph{every} bijection $\theta: \N \rightarrow \mathfrak{A}$, the series $\sum_{j=0}^{\infty} x_{\theta(j)}$ converges in $X$ to $s$.  
\end{prop}
The element $s\in X$   is naturally called the \emph{sum}  of the  series, and we write 
		$s=\sum_{\alpha \in \mathfrak{A}} x_{\alpha}$. 
\begin{proof}
By definition, there exists a bijection $\sigma: \N \rightarrow \mathfrak{A}$,  such that for each continuous seminorm $p$ on $X$, the series $\sum_{j=0}^{\infty} p(x_{\sigma(j)})$ converges. Let $y_j = x_{\sigma(j)}$ and $s_k = \sum_{j=0}^k y_j$. Since $\sum_{j=0}^{\infty} p(y_j)$ converges, for $\epsilon > 0$ there exists $N_0 \in \N$ such that whenever $m, \ell \in \N$ with $m \geq \ell \geq N_0$, $\sum_{j = \ell+1}^{m} p(y_j) < \epsilon$. Therefore for $m \geq \ell \geq N_0$,
\begin{equation}\label{eq:4}
	p(s_m - s_{\ell}) = p \Big (\sum_{j=\ell+1}^{m} y_j \Big ) \leq \sum_{j=\ell +1}^{m} p(y_j) < \epsilon.   \end{equation}
Therefore $\{s_k\}$ is Cauchy sequence in the complete LCTVS $X$, and therefore converges to an $s\in X$.
 In order to complete the proof, it suffices to show that for every bijection $\tau: \N \rightarrow \N$, the series $\sum_{j=0}^{\infty} y_{\tau(j)}$ converges to the same sum $s$. Let $s_k^{\tau} = \sum_{j=0}^{k} y_{\tau(j)}$. Choose $u \in \N$ such that the set of integers $\{0, 1,2,\cdots,N_0\}$ is contained in the set $\{\tau(0), \tau(1), \cdots, \tau(u)\}$. Then, if $k > u$, the elements $y_1, \cdots, y_{N_0}$ get cancelled in the difference $s_k - s_k^{\tau}$ and we have $p(s_k - s_k^{\tau}) < \epsilon$ by \eqref{eq:4}. This proves that the sequences $\{s_k\}$ and $\{s_k^{\tau}\}$ converge to the same limit. So, $s_k^{\tau} \rightarrow s$ as $k \rightarrow \infty$.
\end{proof}

By a famous theorem of Dvoretzky and Rogers, the converse of the above result fails when $X$ is an infinite dimensional Banach space (\cite[Chapter 4]{kadets1997}).

We will also need the notion of the weak (or Pettis) integral of a LCTVS valued function which we will now recall (see  \cite[p. INT III.32-39]{bourbaki} for details).
 Let $K$ be a compact Hausdorff space and let $\mu$ be a Borel measure on $K$, and let $f$ be a continuous map from $K$ to an LCTVS $X$. 
 An element $x \in X$ is called a \emph{Pettis integral} of $f$ on $K$ with respect to $\mu$ if for all $\phi \in X'$,
\begin{equation}\label{eq-pettis}
\phi (x) =  \int_{K} (\phi \circ f)\,d\mu,
\end{equation}
where $X'$ denotes the dual space of  $X$ (the space of continuous linear functionals on $X$), and the right hand side of \eqref{eq-pettis} is an integral of a continuous function. If $X$ is  complete, one can show that there exists a unique $x \in X$ such that \eqref{eq-pettis} holds and we denote the Pettis integral of $f$ on $K$ with respect to $\mu$ by $\displaystyle{\int_K f \hspace{1mm} d\mu=x} $. In fact, the integral exists uniquely, as soon as the space $X$ is \emph{quasi-complete}, i.e. if each \emph{bounded} Cauchy net in $X$ converges, where a net $\{x_{\alpha}\}_{\alpha \in \Gamma}$
 is bounded if for each continuous seminorm $p$,  the net of real numbers  $\{p(x_\alpha)\}_{\alpha\in \Gamma}$ is bounded. While there are situations in which this more refined existence theorem for Pettis integrals is useful (e.g. when the space $X$ is a dual space with weak-* topology), in this paper, we only consider integrals in complete LCTVS. Since each Hausdorff TVS has a unique Hausdorff
 completion, we can define Pettis integrals in any LCTVS, provided we allow the integral to have a value in the completion.

 If $T:X\to Y$ is a continuous linear map of  complete LCTVSs, and $f:K\to X$ is continuous then we have
\begin{equation}\label{eq-t}
 T\left(\int_K f \, d\mu \right)= \int_K T\circ f \, d\mu,
 \end{equation}
since for any linear functional $\psi\in Y'$,
\[\psi\left( 
 T\left(\int_K f \, d\mu \right)\right) = (\psi\circ 
 T)\left(\int_K f \, d\mu \right)= 
\int_K (\psi\circ T)(f) \, d\mu  \]
using the fact that $\psi\circ T\in X'$ and the definition \eqref{eq-pettis} of the Pettis integral.

\subsection{Representation of the torus on an LCTVS $X$.}\label{action of torus}
Let $$\T^n = \{\lambda=(\lambda_1,...,\lambda_n) \in \C^n : \abs{\lambda_j} = 1 \hspace{1mm} \text{for every $1 \leq j \leq n$}\}$$
be the $n$-dimensional unit torus. With the subspace topology
inherited from $\cx^n$ and  binary operation defined as $ (\lambda , \xi ) \mapsto \lambda {\cdot} \xi = (\lambda_1 \xi_1,\cdots,\lambda_n \xi_n)$, is a compact abelian topological group. For a LCTVS $X$, and a continuous function $f:\T^n \to X$, we denote 
the Pettis  integral of $f$ with respect to the Haar measure of $\T^n$ (normalized to be a probability measure)
by
\[ \int_{T^n} f(\lambda)d\lambda.\]

 Let $X$ be an LCTVS, and let $\lambda\mapsto \sigma_\lambda$ be a continuous representation of $\T^n$ on $X$. Recall that this means that  for each $\lambda\in \T^n$, the map $\sigma_\lambda$ is an automorphism (i.e. linear self-homeomorphism) of $X$ as a topological vector space, the map  $\lambda\mapsto \sigma_\lambda$ is a group homomorphism from $\T^n$ to $\mathrm{Aut}(X)$, and the associated map
\[ \sigma: \T^n\times X\to X, \quad \sigma(\lambda, x)=\sigma_\lambda (x)\]
is continuous.


Given a  representation  $\sigma$ of the group  $\T^n$ on an LCTVS $X$,  a continuous seminorm $p$ on $X$ is said 
to be invariant (with respect to $\sigma$)  if $p(\Sl(x)) = p(x)$ for all $x \in X$ and $\lambda \in \T^n$.
 \begin{prop}\label{prop-continuity}
A representation $\sigma$ of $\T^n$ on an LCTVS $X$ is continuous if and only if the following two conditions are both satisfied:
  \begin{enumerate}[label=(\alph*)]
  	\item the topology of $X$ is generated by a family of invariant seminorms, and
  \item for each $x \in X$ the function from $\T^n$ to $X$ given by $\lambda \mapsto \Sl(x)$ is continuous at the identity element of $\T^n$.
  \end{enumerate}
 \end{prop}
 \begin{proof}
 Assume that $\sigma$ is continuous, i.e.,  $\sigma:\T^n\times X\to X$ is continuous, so for $x \in X$, the function $\lambda \mapsto \Sl(x)$ is continuous on $\T^n$, in particular at the identity.  Therefore (b) follows.
 
 For a continuous seminorm $q$ on $X$, define $p(x)=\sup_{\lambda \in \T^n} q(\Sl(x))$, which is  finite since $\T^n$ is compact, and which is easily seen to be a  seminorm. To show $p$ is invariant, we note that
  \[ p(\Sl(x)) = \sup \limits_{\mu \in \T^n} q(\sigma_\mu(\Sl(x))) = \sup \limits_{\mu \in \T^n} q(\sigma_{\mu{\cdot} \lambda}(x)) =  \sup \limits_{\xi \in \T^n} q(\sigma_{\xi}(x)) = p(x).\]
It remains to show that $p$ is continuous.   For $x,y\in X$, we have
\[p(x)=\sup_{\lambda\in \T^n}q\left(\sigma_\lambda(y)+ \sigma_\lambda(x-y)\right) \leq  \sup_{\lambda\in \T^n}\left(q(\sigma_\lambda(y))+ q(\sigma_\lambda(x-y))\right)\leq p(y)+p(x-y),\]
so that $\abs{p(x) - p(y)} \leq p(x-y)$ for all $x,y \in X$, and it follows that the seminorm $p$ is continuous on $X$ if and only if it is continuous at $0$. We show that for each $\epsilon>0$,  there exists a neighbourhood $V$ of $0$ in $X$ such that for all $\lambda \in \T^n$, $q \circ \Sl < \epsilon$ on $V$. For each  $\xi \in \T^n$, since $q$ and $\sigma$ are continuous, there exists a neighborhood $U_{\xi}$ of $\xi$  in $\T^n$ and a neighbourhood $V_{\xi}$ of $0$ in $X$ such that
$ q(\Sl(x)) < \epsilon$  for all $x \in V_{\xi}$ and $\lambda \in U_{\xi}$.
The collection $\{U_{\xi}\}_{\xi \in \T^n}$ forms an open cover of $\T^n$. Since $\T^n$ is compact, let $\{U_{\xi_1}, ..., U_{\xi_k}\}$ be a finite subcover of $\T^n$ corresponding to the open cover.  Then for all $x \in  \bigcap_{j=1}^{k}V_{\xi_j}  $ and $\lambda \in \T^n$, we have $q(\Sl(x)) < \epsilon$.

 Now assume the two conditions (a) and (b).  Let $(\Gamma, \geq)$ be a directed set and let $(\lambda_\alpha)_{\alpha \in \Gamma}$ and $(x_\alpha)_{\alpha \in \Gamma}$ be nets in $\T^n$ and $X$ respectively with $(\lambda_{\alpha}, x_\alpha) \rightarrow (\lambda, x)$ in $\T^n \times X$. We need to show $\Slalpha (x_\alpha) \rightarrow \Sl(x)$ in $X$, i.e., $p(\Slalpha (x_\alpha)  -\Sl(x))\to 0$ for each invariant continuous seminorm $p$ of $X$. But we have, by the invariance of $p$:
 \[ p\left(\Slalpha (x_\alpha\right)  -\Sl(x))= p\left(x_\alpha-\sigma_{\lambda_\alpha^{-1}\lambda}(x)\right)\leq p\left(x_\alpha-x\right)+ p\left(x-\sigma_{\lambda_\alpha^{-1}\lambda}(x)\right) .\]
 The term $p\left(x_\alpha-x\right)$  goes to zero since $x_\alpha\to x$, and the term $p\left(x-\sigma_{\lambda_\alpha^{-1}\lambda}(x)\right) $
 also  goes to zero since $\lambda_\alpha\to \lambda$ and $\mu \mapsto \sigma_\mu(x)$ is 
 continuous at $\mu=\mathbf{1}$. The result follows.
 \end{proof}
 
 \subsection{Abstract Fej\'{e}r Theorem}\label{abstract Fejer theorem}
Let $X$ be a complete LCTVS and let $\sigma$ be a continuous representation of $\T^n$ on $X$. For each $\alpha = (\alpha_1,...,\alpha_n) \in \Z^n$ and 
$x\in X$, define
\begin{equation}\label{eq-fouriermode}
 \pialpha^{\sigma}(x) = \int_{\T^n} \lambda^{-\alpha} \sigma_{\lambda} (x) \,d \lambda,   
\end{equation}
the Pettis integral of the continuous function $\lambda \mapsto \lambda^{-\alpha} \sigma_{\lambda} (x)$ on $\T^n$ with respect to the Haar 
probability measure of $\T^n$.  We will say that $\ppi^\sigma_\alpha(x)$ is the $\alpha$-th \emph{Fourier component} of $x$ with respect to the representation $\sigma$.  We use the standard convention with respect to multi-index powers, i.e., $\lambda^\alpha=\lambda_1^{\alpha_1}\dots \lambda_n^{\alpha_n}$.  We will say that the subspace of $X$ defined as
\begin{equation}\label{eq-modedef}
 [X]^\sigma_{\alpha}=\{x\in X: \sigma_\lambda(x)=\lambda^\alpha\cdot x \text{ for all } \lambda\in \T^n \}
\end{equation}
is the $\alpha$-th \emph{Fourier mode} of the space $X$, and we will call the map $\ppi^\sigma_\alpha$
the $\alpha$-th  \emph{Fourier projection}, both with respect to the representation $\sigma$. We note the following facts:
\begin{prop}\label{prop-fourierproperties}
	As above, let $X$ be a complete LCTVS and $\sigma$ be a continuous representation of $\T^n$ on $X$. 
\begin{enumerate}
	\item For each $\alpha\in \Z^n$, the $\alpha$-th Fourier mode $[X]^\sigma_\alpha$
	 is a closed $\sigma$-invariant  subspace of $X$, and the Fourier projection $\ppi^\sigma_\alpha$ is a continuous linear projection from 
	 $X$ onto  $[X]^\sigma_\alpha$. 
	 \item Let $Y$ be another complete LCTVS,  and let $\tau$ be a continuous representation of $\T^n$ on $Y$, and let 
	  $j:Y\to X$ be a continuous 
	 linear map intertwining $\sigma $ and $\tau$, i.e., for each $\lambda\in \T^n$, 
	 $ j\circ \tau_\lambda = \sigma_\lambda\circ j.$
	 Then for each  $\alpha\in \Z^n$, we have
	 \begin{equation}\label{eq-equivariance}
	 	 j\circ \ppi^\tau_\alpha=\ppi^\sigma_\alpha\circ j.
	 \end{equation}
\end{enumerate}
\end{prop}
\begin{proof}\begin{enumerate}[wide]
		\item 
		Linearity of $\pialphasigma$ follows from the linearity of $\Sl$. Recall from Proposition \ref{prop-continuity} that there exists a family $\mathscr{P}$ of continuous invariant seminorms that generates the locally convex topology of $X$. To see the continuity of $\pialphasigma$, observe that for each $p\in \mathscr{P}$, we have
		\begin{equation}\label{eq:2E}
			p (\pialphasigma (x)) = p \all{ \int_{\T^n} \lambda^{-\alpha} \sigma_{\lambda} (x) \,d \lambda  }
			\leq  \int_{\T^n}p\all{ \lambda^{-\alpha} \sigma_{\lambda} (x) } \,d \lambda = \int_{\T^n} p (x) \,d \lambda  = p (x),
		\end{equation}
		where the inequality is due to Proposition 6 in \cite[p. INT III.37]{bourbaki}. Now, if $x \in [X]^\sigma_\alpha$, then
		\begin{equation*}
			\pialphasigma (x) = \int_{\T^n} \lambda^{-\alpha} \sigma_{\lambda} (x) \hspace{1mm}d \lambda =  \int_{\T^n} \lambda^{-\alpha} \cdot \lambda^{\alpha} x  \hspace{1mm}d \lambda = x \cdot \bigg( \int_{\T^n} \hspace{1mm}d \lambda \bigg) = x,
		\end{equation*}
		so that $[X]^\sigma_\alpha\subset \pialphasigma(X)$. To prove that $\pialphasigma(X) \subset [X]^\sigma_\alpha$, notice that 
		for each $\alpha\in \Z^n$, each $\lambda\in \T^n$ and
		$x\in X$, 
		\begin{align*}
			\sigma_\lambda(\ppi^\sigma_\alpha(x))&= \sigma_\lambda\left(\int_{\T^n} \mu^{-\alpha} \sigma_\mu(x)d\mu\right)
			= \int_{\T^n} \mu^{-\alpha} \sigma_{\lambda\cdot\mu}(x)d\mu& \text{using \eqref{eq-t}}\nonumber\\
			&=\lambda^\alpha \cdot\int_{\T^n}(\lambda\mu)^{-\alpha} \sigma_{\lambda\mu}(x)d\mu
			=\lambda^\alpha  \cdot\ppi^\sigma_\alpha(x)& \text{by Haar invariance}.\label{eq-into}
		\end{align*}

	\item For $x\in X$, we have
		\[ j\circ \ppi^\tau_\alpha(x) = j\left( \int_{\T^n}\lambda^{-\alpha}\tau_\lambda(x)d\lambda\right)= \int_{\T^n}\lambda^{-\alpha}j\circ\tau_\lambda(x)d\lambda= \int_{\T^n}\lambda^{-\alpha}\sigma_\lambda\circ j(x) d\lambda= \ppi^\sigma_\alpha\circ j(x),\]
		where we have used \eqref{eq-t} to go from the second to the third step.  
	\end{enumerate}
\end{proof}
\begin{rem}
	The inequality \eqref{eq:2E}
	\begin{equation}\label{eq-cauchy}
	p \left(\pialphasigma (x)\right)\leq p(x),
	\end{equation}
which holds for each $\alpha$ for an invariant seminorm $p$ can be thought of as an abstract
form of the familiar \emph{Cauchy inequalities} of complex analysis.
\end{rem}

If  $X$  is $L^1(\T)$, the Banach space of integrable functions  on 
$\T$, and if $\sigma$ is the continuous representation of $\T$ on $L^1(\T)$
given by 
\[ \sigma_\lambda(f)(\mu)= f(\lambda\cdot\mu), \quad \lambda, \mu\in \T,  f\in L^1(\T),\]
an easy computation shows that for $\phi\in \rl$, and $ f\in L^1(\T)$, we have
\[  \pialpha^{\sigma}(f)(e^{i\phi})=e^{i\alpha\phi}\cdot\hat{f}(\alpha),\quad\text{with } \hat{f}(\alpha)= \frac{1}{2\pi}\cdot\int_0^{2\pi}e^{-i\alpha\theta}f( e^{i\theta})d\theta,\]
the $\alpha$-th term of the Fourier series of $f$. 

It is therefore natural to define, for $x \in X$, the \emph{Fourier series} of $x$ with respect to $\sigma$ to be the formal series 
\begin{equation}\label{eq-fourierseries}
x \sim \sum \limits_{\alpha \in \Z^n} \pialpha^{\sigma} (x).
\end{equation}
For an integer $k\geq 0$, define the $k$-th \emph{square partial sum} of the Fourier series in \eqref{eq-fourierseries} by
\begin{equation}\label{eq:squaresum}
S_k^{\sigma}(x)= \sum \limits_{\abs{\alpha}_{\infty} \leq k} \pialpha^{\sigma} (x),
\end{equation}
where $\abs{\alpha}_{\infty}  \coloneqq \max \big \{\abs{\alpha_j}, 1 \leq j \leq n \big \}$. We are ready to state an abstract version of Fej\'{e}r's theorem. 
\begin{thm}\label{thm-fejer}
Let $\sigma$ be a continuous representation of $\T^n$ on an LCTVS $X$ and let $x \in X$. Then the Ces\`{a}ro means of the square partial sums of the Fourier series of $x$ (with respect to $\sigma$) converge to $x$ in the topology of $X$.
\end{thm} 
\begin{proof}
Write the Ces\`{a}ro means of the square partial sums of the Fourier series of $x$ as,
\begin{align*} C_N^{\sigma}(x) &=  \frac{1}{N+1} \sum \limits_{k=0}^{N} S_k^{\sigma} (x) = \frac{1}{N+1} \sum \limits_{k=0}^{N} \sum \limits_{\abs{\alpha}_{\infty} \leq k} \pialphasigma(x)= \frac{1}{N+1} \sum \limits_{k=0}^{N} \sum \limits_{\abs{\alpha}_{\infty} \leq k} \left( \int \limits_{\T^n} \lambda^{-\alpha} \Sl(x)\, d \lambda\right)\\
&=\frac{1}{N+1} \int \limits_{\T^n} \sum \limits_{k=0}^{N} \left( \sum \limits_{\abs{\alpha}_{\infty} \leq k} \lambda^{-\alpha} \right) \Sl(x) \, d\lambda  = \int \limits_{\T^n} F_N (\lambda) \hspace{1mm} \Sl(x) \, d\lambda 
\end{align*}
where $$ F_N (\lambda) =\displaystyle{ \frac{1}{N+1}  \sum \limits_{k-0}^{N} \left( \sum \limits_{\abs{\alpha}_{\infty} \leq k} \lambda^{-\alpha} \right)}$$ is the classical
 Fej\'{e}r kernel.  Introducing polar coordinates, $\lambda_j=e^{i\theta_j}$ on $\T^n$, and summing, we obtain the classical representation
 \[ F_N(e^{i\theta_1}, \cdots , e^{i\theta_n}) =\frac{1}{N+1}\prod_{j=1}^{n} \frac{\sin^2 \left(\left(\frac{N+1}{2} \right) \theta_j \right)}{\sin^2  
 \left(\frac{\theta_j}{2} \right) }\]
 It is well-known that the Fej\'{e}r kernel has the properties that 
 \begin{itemize}
     \item[(a)]$F_N \geq 0$ for all $N$,
     \item[(b)] $\int \limits_{\T^n} F_N(\lambda)\, d\lambda = 1$
and 
\item[(c)]For each $\delta > 0$, $F_N \rightarrow 0$ uniformly on $\T^n \setminus B(\mathbf{1}, \delta) $, where $B(\mathbf{1}, \delta)$ is the $n$-dimensional ball centered at $\mathbf{1}= (1,1,...,1)$ and radius $\delta$.
 \end{itemize}
Let $p$ be a continuous  $\sigma$-invariant seminorm on $X$. Then for $x\in X$, we have
\begin{align}
p(C_N(x) - x) &= p \left( \int \limits_{\T^n} F_N (\lambda) \cdot \Sl(x) \, d\lambda - x \cdot \int \limits_{\T^n} F_N(\lambda) \, d\lambda \right) \nonumber\\ & 
\qquad\text{using property (b) of $F_N$} \nonumber \\
&= p \left( \int \limits_{\T^n} F_N (\lambda) \cdot \Sl(x) \,d\lambda - \int \limits_{\T^n} x \cdot F_N(\lambda) \,d\lambda \right) \nonumber \\
& \qquad \text{using \eqref{eq-t}}\nonumber\\
&= p \left( \int \limits_{\T^n} \big (\Sl(x) - x \big ) \cdot F_N(\lambda) \, d \lambda\right)\nonumber  \\
&\leq \int \limits_{\T^n}  p (\Sl(x) -x)\cdot F_N(\lambda) \, d \lambda \label{eq:35}\\
& \qquad \text{ using  \cite[Prop. 6, p. INT III.37]{bourbaki} and the positivity of $F_N(\lambda)$}\nonumber
\end{align}
Since $\lambda \mapsto \Sl(x)$ is continuous and $p$ is a continuous seminorm, there exists $\delta >0$ such that
$p (\Sl(x) - \sigma_{\mathbf{1}}(x)) = p(\Sl(x) - x) < \frac{\epsilon}{ 2} $  whenever $\lambda \in \T^n \cap B(\mathbf{1},\delta).$
Then on the set  $\T^n \setminus B(\mathbf{1},\delta)$ we have
$$
p(\Sl(x) - x) \leq p(\Sl(x)) + p(x) = 2 \cdot p(x).
$$
By property (c) of $F_N$,   there exists $N_1 \in \N$ such that whenever $\lambda \in \T^n \setminus B(\mathbf{1},\delta)$,  for all $N \geq N_1$ we have,
\begin{equation} \label{eq:37}
 F_N (\lambda) < \frac{\epsilon}{4 \cdot p(x)}   
\end{equation}
Now from \eqref{eq:35}, we have
\begin{align*}
 p (C_N(x) - x) 
 &\leq \int \limits_{\T^n \cap B(\mathbf{1},\delta)}   F_N(\lambda) \cdot p (\Sl(x) -x) \, d \lambda + \int \limits_{\T^n \setminus B(\mathbf{1},\delta)} F_N(\lambda) \cdot p (\Sl(x) -x) \, d \lambda \\
 & < \frac{\epsilon }{2} \int \limits_{\T^n} F_n(\lambda) \, d\lambda + \frac{\epsilon}{4 \cdot p(x)} \cdot 2 \cdot p(x) \\
 &= \frac{\epsilon}{2} +  \frac{\epsilon}{2} = \epsilon.
\end{align*}
Since by Proposition~\ref{prop-continuity} the topology of $X$ is generated by $\sigma$-invariant seminorms, the result follows.
\end{proof}
\begin{cor}\label{cor-cesaro}
    Let $X$ be a complete LCTVS and suppose that we are given
    a continuous representation of the group $\T^n$ on $X$. Then if the Fourier series of an element $x\in X$ with respect to this representation is absolutely convergent in $X$, the sum of the Fourier series equals $x$.
\end{cor}
\begin{proof}
Since the series $\sum_{\alpha \in \Z^n} \pi_{\alpha}^{\sigma}(x)$ is absolutely convergent, by Proposition \ref{prop-absolute}, there exists the sum of the series, i.e., an  $\wt{x} \in X$  such that for every bijection $\theta:\N \rightarrow \Z^n$ we have $ \sum_{j=0}^{\infty} \pi_{\theta(j)}^{\sigma}(x) = \wt{x}$. Let $S_N = \sum_{j=0}^{N} \pi_{\theta(j)}^{\sigma}(x)$; then the sequence of partial sums $\{S_N\}$ converges to $\wt{x}$ in $X$. By Proposition \ref{prop-cesaro}, the sequence of Ces\`{a}ro means $\{C_N\}$ of the partial sums converges to $\wt{x}$ as well. However, by Theorem \ref{thm-fejer}, the Ces\`{a}ro means converge to $x$. Therefore $\wt{x} = x$.     
\end{proof}

\section{Recapturing Complex analysis}\label{sec-distn}
	  We now use the machinery developed in the previous section to give a conceptually simple account of the remarkable regularity properties of holomorphic distributions.
So we will pretend that we have forgotten everything about complex analysis, but do remember the rudiments of the  theory of distributions, accounts of which can be found in the classic
 treatises \cite{horvol1,treves,schwartz}. First we  clarify notations and recall a few facts.
 
 \subsection{The basic spaces} For an open $\Omega\subset\rl^n$, the space $\DD(\Omega)$ of test functions is the $LF$-space of smooth compactly supported complex valued functions, topologized as the inductive limit of 
 the Fréchet spaces $\DD_K$ consisting, for a given compact $K\subset \Omega$, of those elements of $\DD(\Omega)$ which have support in $K$. 
 Recall that a subset $B\subset \DD(\Omega)$ is bounded, if  and only if there is a compact $K \subset \Omega$ such that 
 $B\subset \DD_K$, and for each nonnegative multi-index $\alpha\in \N^n$, we have
 \begin{equation}
 	\label{eq-bounded}
  \sup_{\substack{\phi\in B\\x\in K}}\abs{\frac{\partial^\alpha}{\partial x^\alpha}\phi(x)}<\infty,
 \end{equation}
where, here and later, we will use standard multi-index conventions such as 
\begin{equation}\label{eq-convention}
\frac{\partial^\alpha}{\partial x^\alpha}=\frac{\partial^{\abs{\alpha}}}{\partial x_1^{\alpha_1}\dots \partial x_n^{\alpha_n}}, \quad \abs{\alpha} =\sum_j \alpha_j, \quad \text{for} \quad \alpha\in \N^n. 
\end{equation}

  The space of distributions  $\DD'(\Omega)$  on $\Omega$ is
 the dual  of $\DD(\Omega)$, consisting of continuous linear forms on $\DD(\Omega)$.
  We denote the value of a distribution $T\in \DD'(\Omega)$ at a test function $\phi\in \DD(\Omega)$ by 
  $\ipr{T,\phi}.$   The space   $\DD'(\Omega)$ is endowed  with the  usual  strong dual topology. Recall that this topology is generated by 
   the  family of seminorms 
   \begin{equation}
   	\label{eq-seminorm}
    \left\{p_B : B\subset \DD(\Omega)  \text { is bounded}\right\}, \quad \text{where,  }  p_B(T)= \sup_{\phi\in B} \abs{\ipr{T,\phi}}.
\end{equation}
  In this topology, the space $\DD'(\Omega)$ is complete. 
  
  Given a locally integrable function $f\in L^1_{\mathrm{loc}}(\Omega)$,  we can associate to $f$ a distribution $T_f$ defined by
  \[ \ipr{T_f, \phi}=\int_\Omega f\phi\, dV, \textrm{ for } \phi\in \DD(\Omega), \]
	  where $dV$ denotes the Lebesgue measure of $\rl^n$.  Then the \emph{locally integrable distribution}  $T_f$ is said to be generated by $f$, and
  as usual we grant ourselves the the right to abuse language by identifying the distribution $T_f\in \DD'(\Omega)$ with the function $f\in L^1_{\mathrm{loc}}(\Omega)$.

We will use the abbreviations 
\begin{equation}\label{eq-dbardef}
 \dd_j= \frac{\partial}{\partial z_j}, \quad  \db_j=  \frac{\partial}{\partial\ol{z_j}}, \end{equation}
for the basic constant coefficient differential operators of complex analysis, acting on functions or distributions on $\cx^n$.
If $\Omega\subset \cx^n$ is an open set, define
\begin{equation}
	\label{eq-holdistn}\mathscr{O}(\Omega)= \left\{ f\in \DD'(\Omega): \db_j f=0, 1\leq j\leq n\right\},
\end{equation}
the space of \emph{holomorphic distributions} on $\Omega$. The subspace $\mathscr{O}(\Omega)$ is closed in the space $\DD'(\Omega)$  by the continuity of the operators $\db_j$, and is therefore a complete LCTVS in the subspace topology.
\subsection{The main theorem}We begin with some definitions and notational conventions. For $\lambda\in \T^n$  we denote by $\rot_\lambda$ the \emph{Reinhardt rotation} of $\cx^n$ by the element $\lambda$, the linear
automorphism of the vector space $\cx^n$ given by
\begin{equation}
	\label{eq-rot}
	\rot_\lambda(z)=(\lambda_1z_1,\dots, \lambda_n z_n).
\end{equation}
A domain $\Omega\subset \cx^n$ is defined to be \emph{ Reinhardt}  if and only if $\rot_\lambda(\Omega)=\Omega$ for each $\lambda\in \T^n$.
Let $Z$ denote the union of the coordinate hyperplanes of $\cx^n$:
\begin{equation}\label{eq-Z}
	Z=\bigcup_{j=1}^n \{z\in \cx^n: z_j=0\}.
\end{equation}
	Recall that a Reinhardt domain $\Omega\subset \cx^n$ is said to be \emph{log-convex}, if whenever $z,w\in \Omega\setminus Z$, the point 
	$\zeta\in \cx^n$  belongs to $\Omega$ if there is a $t\in [0,1]$ such that  for $1\leq j\leq n$
	\begin{equation}\label{eq-logconvex1}
		 \abs{\zeta_j}= \abs{z_j}^t \abs{w_j}^{1-t}.
	\end{equation}
For $\alpha\in \Z^n$, let  $e_\alpha$ be the monomial function given by 
\begin{equation}
	\label{eq-ealpha} e_\alpha(z)=z^\alpha=z_1^{\alpha_1}\dots z_n^{\alpha_n}.
\end{equation}
For a Reinhardt subset $\Omega\subset \cx^n$ and $1\leq j \leq n$, let
\begin{equation}\label{eq-relcomp}
	{\Omega}^{(j)}=\{(z_1,\dots, \zeta z_j, \dots, z_n): z\in \Omega, \zeta\in \ol{\D}\},
\end{equation}
where $\ol{\D}=\{\abs{\zeta}\leq 1\}\subset \cx$ is the closed disk. This can be thought of as the result of ``completing" $\Omega$ in the $j$-th coordinate. 
Following \cite{jarpflubook}, we say that the Reinhardt domain $\Omega\subset \cx^n$ is \emph{relatively  complete} if for each $1\leq j \leq n$, whenever we have 
$\Omega\cap\{z_j=0\}\not=\emptyset$, we also have
\[ \Omega^{(j)}\subset \Omega.\]
We prove the following well-known structure theorem for holomorphic distributions on Reinhardt domains, as an application of the ideas of Section~\ref{sec-abstract}:
\begin{thm}\label{thm-distn}
	Let $\Omega $ be a Reinhardt domain in $\cx^n$ and let 
	\begin{equation}\label{eq-somega}
		 \mathcal{S}(\Omega)= \{ \alpha\in \Z^n: e_\alpha \in \mathcal{C}^\infty(\Omega)\}.
	\end{equation}
Let $\wh{\Omega}$ be the smallest relatively complete log-convex Reinhardt domain in $\cx^n$ that contains $\Omega$.
The  for each $\alpha\in \mathcal{S}(\Omega)$ 
there is a continuous linear functional 
	$a_\alpha:\OO(\Omega)\to \cx$ such that for each $T\in \OO(\Omega)$, the series 
	\begin{equation}\label{eq-laurent}
\sum_{\alpha\in \mathcal{S}(\Omega)} a_\alpha(T)e_\alpha
	\end{equation}
	converges absolutely in $\mathcal{C}^\infty(\wh{\Omega})$  to a function $f\in \mathcal{C}^\infty(\wh{\Omega})$, and  $f|_\Omega$ generates 
	the distribution $T$.
\end{thm}
\textbf{Remarks:}\begin{enumerate}[wide]
\item For $n=1$, all Reinhardt domains in the plane, i.e., disks and annuli, are automatically relatively complete and log-convex. For $n\geq 2$, 
it is easy to give examples of Reinhardt domains which are not log-convex, or not relatively complete, or perhaps both. For such a domain $\Omega$, it follows that each holomorphic
distribution $f\in \OO(\Omega)$ extends to a holomorphic function $F\in \OO(\wh{\Omega})$. This is the simplest example of \emph{Hartogs phenomenon}, the compulsory extension of
all holomorphic functions from a smaller domain to a larger one, characteristic of  domains in several complex variables. 
\item  	The functionals $a_\alpha$ are called the \emph{coefficient functionals}, and the series \eqref{eq-laurent} is  of course  the \emph{Laurent series}
of the function $f$ (the \emph{Taylor series} if $\mathcal{S}(\Omega)=\N^n$).
\item It is known (by  a direct construction of a plurisubharmonic exhaustion)  that relatively complete log-convex Reinhardt domains are pseudoconvex. This means that such a domain $\Omega$
admits a holomorphic distribution whose Laurent expansion converges absolutely precisely on $\Omega$.
\end{enumerate}

As an immediate consequence of Theorem~\ref{thm-distn}  we have the following:
\begin{cor}\label{cor-complexanalytic}
Let $\Omega\subset \cx^n$ be open. Then each distribution $T\in \OO(\Omega)$ is  \emph{complex-analytic}, i.e., for each $p\in \Omega$ there
 is a neighborhood $U$ of $p$, where the function $f$ generating $T$  is represented by a Taylor series centered at $p$.
\end{cor}
\subsection{Holomorphic  functions and maps} 
\label{sec-holfn}

A holomorphic distribution $T\in \OO(\Omega)$  will be called a \emph{holomorphic function} if it is generated by a $\mathcal{ C}^\infty$ function $f$.   We denote the space of holomorphic functions on $\Omega$  temporarily by  $(\OO\cap\mathcal{C}^\infty)(\Omega)$. Once Theorem~\ref{thm-distn} is proved, it will  follow that  $(\OO\cap\mathcal{C}^\infty)(\Omega)=\OO(\Omega)$.

Let $\Omega_1, \Omega_2$ be domains in $\cx^n$.  By a \emph{holomorphic map} $\Phi:\Omega_1\to \Omega_2$, we mean a mapping each of 
whose components is a holomorphic function on $\Omega_1$.  A holomorphic map is  a \emph{biholomorphism}, if it is a bijection, and its set-theoretic inverse is also a holomorphic map. (It is of course known that the assumption of the holomorphicity of the inverse map is redundant, but this is a consequence of complex-analyticity, which is exactly what we are proving here).
If $\Phi: \Omega_1\to \Omega_2$ is a biholomorphism, then for a distribution $T\in \DD'(\Omega_2)$, we can define in the usual way the \emph{pullback} distribution $\Phi^*T\in \DD'(\Omega_1)$:  if $T$ is generated by a test function $f\in \DD(\Omega_2)$, then 
$\Phi^*T$ is the distribution generated by the function $f\circ\Phi$, and for general $T$, we extend this definition by continuity, using the density of
test functions  in $\DD'(\Omega)$, see \cite[Theorem~6.1.2]{horvol1} for details. Extending the chain rules for the complex 
derivative operators  from test functions to distributions, we have the following relations analogous to  \cite[formula (6.1.2)]{horvol1} for the Wirtinger derivatives \eqref{eq-dbardef}:
\begin{equation}\label{eq-chain1}
\dd_j(\Phi^* T)= \sum_{k=1}^n \dd_j{\Phi_k}\cdot \Phi^*\left(\dd_k  T\right),
\end{equation}
and
\begin{equation}\label{eq-chain2}
\db_j(\Phi^* T)= \sum_{k=1}^n \db_j\ol{\Phi_k}\cdot \Phi^*\left( \db_k T\right),
\end{equation}
where as above, $\Phi:\Omega_1\to \Omega_2$ is a biholomorphism of domains in $\cx^n$, written in components as $\Phi=(\Phi_1,\dots, \Phi_n)$, 
and $T\in \DD'(\Omega_2)$.

Therefore, we have the following immediate consequence of \eqref{eq-chain2}: 
\begin{prop}\label{prop-invariance}
If $\Phi:\Omega_1\to \Omega_2$ is a biholomorphism and $T\in \OO(\Omega_2)$, then  we have $\Phi^*T\in \OO(\Omega_1)$. If $f\in (\OO\cap \mathcal{C}^\infty)(\Omega_2)$, then $\Phi^*f\in (\OO\cap \mathcal{C}^\infty)(\Omega_1)$.
\end{prop}
Therefore the spaces of holomorphic distributions and functions 
are invariant under pullbacks under biholomorphic maps.
In fact, in the proof of Theorem~\ref{thm-distn}, only two  simple special cases of Proposition~\ref{prop-invariance} noted below are needed:
\begin{enumerate}[wide]
	
	\item Translation by a vector $a\in \cx^n$ is the map $\tra_a:\cx^n\to \cx^n$ 
	\begin{equation}\label{eq-tra}
		\tra_a(z)=z+a, 
	\end{equation}
	which is obviously a biholomorphic automorphism of $\cx^n$. For a domain $\Omega\subset \cx^n$, we therefore have a pullback isomorphism of spaces of holomorphic distributions
	$\tra_a^*:\OO( \tra_a (\Omega))\to \OO(\Omega)$. This can be thought of as an expression of the fact that the operator $\dbar= (\db_1,\dots, \db_n)$ is translation invariant. 

	\item The Reinhardt rotations $\rot_\lambda$ of \eqref{eq-rot} are clearly biholomorphic automorphisms of Reinhardt domains, and the pullback operation induces a representation 
	of the group $\T^n$ on the space of holomorphic distributions (see \eqref{eq-sigmadef} below). 
	
	 A domain $\Omega$ is said to be
	\emph{Reinhardt centered at $a$} for an $a\in \cx^n$ if there is a Reinhardt domain $\Omega_0$ such that $\Omega= \tra_a(\Omega_0)$. 
Every open set in $\cx^n$ 
has local Reinhardt symmetry, in the sense that each point has neighborhood which is a Reinhardt domain centered at that point.
\end{enumerate}

\subsection{Mean value property of the monomials} 
It is easily verified by direct computation that for each  $\alpha$, we have  $e_\alpha\in (\OO\cap \mathcal{C}^\infty) (\cx^n\setminus Z)$, where $e_\alpha$ is the monomial of  \eqref{eq-ealpha}.
We now  note a remarkable symmetry property (the Mean-Value Property) of the functions $e_\alpha$:
\begin{lem} \label{lem-radial}Let $\alpha\in \Z^n$,  let $z\in \cx^n\setminus Z$, and  let 
	 $\psi\in \DD(\cx^n\setminus Z)$ be a test function  which has radial symmetry in each variable around the point $z$, i.e., there are functions $\rho_1,\dots, \rho_n\in \DD(\rl)$ such that 
	 $\psi(\zeta) = \prod_{j=1}^{n}\rho_j(\abs{\zeta_j-z_j})$, and whose integral is 1:
	 \begin{equation}\label{eq-int}
	 	 \int_{\cx^n\setminus Z} \psi(\zeta)dV(\zeta)=1.\end{equation}
Then  we have
	\begin{equation}\label{eq-meanvalue}
		\ipr{e_\alpha,\psi}=e_\alpha(z)
	\end{equation}
\end{lem}
\begin{proof} First consider the case $n=1$. 
If $\alpha\geq 0$,  the formula
\begin{equation}\label{eq-mvp1}
	\int_0^{2\pi} (z+re^{i\theta})^\alpha d\theta=2\pi z^\alpha
\end{equation}
	holds for $r>0$, by expanding the integrand using the binomial formula, and integrating the finite sum term by term. The formula \eqref{eq-mvp1} also holds
	for $\alpha<0$, provided $z\not =0$, and $0<r<\abs{z}$. This follows on noticing that we have an infinite series expansion
		\[ (z+re^{i\theta})^\alpha = z^\alpha\left(	1+\frac{r}{z}e^{i\theta}\right)^\alpha= z^\alpha \sum_{k=0}^\infty \binom{\alpha}{k}\left(\frac{r}{z}\right)^ke^{ik\theta},\]
		where by the 
		$M$-test, the convergence is uniform in $\theta$, and then integrating the series on the right term by term. 
		
		Now  $\psi(\zeta)=\rho(\abs{\zeta-z})$ and $\psi$ is  supported in some disc $B(z,R)\subset \cx^*=\cx\setminus\{0\}$, which means $R<\abs{z}$.  Then the normalization \eqref{eq-int} 
		is equivalent to 
		\begin{equation}\label{eq-mvp2}
		2\pi \int_0^R \rho(r) rdr=1.
		\end{equation}
		Now
		\begin{align*}
			\ipr{e_\alpha,\psi}&=\int_{\cx^*} e_\alpha(\zeta)\rho(\abs{\zeta-z})dV(\zeta)
			= \int\limits_{0}^{R}\int\limits_{0}^{2\pi}(z+re^{i\theta})^\alpha \rho(r)rd\theta dr\\
			&=\int_{0}^{R}\rho(r)r\left(\int\limits_{0}^{2\pi}(z+re^{i\theta})^\alpha d\theta\right)dr=2\pi z^\alpha\cdot \int_{0}^{R}\rho(r)r dr &\text{using \eqref{eq-mvp1}}\\
			&=z^\alpha =e_\alpha(z),&\text{using \eqref{eq-mvp2}}.
		\end{align*}
	which establishes the result for $n=1$. In the general case, notice that for $\alpha\in \Z^n$ and $\zeta \in \cx^n\setminus Z$ we have 
	$ e_\alpha(\zeta)= \prod_{j=1}^n e_{\alpha_j}(\zeta_j).$
	Therefore, since $\cx^n\setminus Z= \cx^*\times \dots \times \cx^*$,
	\begin{align*}
		\ipr{e_\alpha,\psi}&=\int_{\cx^n\setminus Z} \prod_{j=1}^n e_{\alpha_j}(\zeta_j)\rho_j(\abs{\zeta_j-z_j})\,dV(\zeta)=  \prod_{j=1}^n \int_{\cx^*} e_{\alpha_j}(\zeta_j)\rho_j(\abs{\zeta_j-z_j})\,dV(\zeta_j)\\
		&=  \prod_{j=1}^n e_{\alpha_j}(z_j)= e_\alpha(z).
	\end{align*}
\end{proof}

\subsection{The representation $\tau$}
If $\Omega\subset \cx^n$ is a Reinhardt domain, then for each $\lambda$, the map $\rot_\lambda$ of \eqref{eq-rot}
 is a biholomorphic automorphism of $\Omega$.
Define a representation  $\tau$ of $\T^n$ on the space $\DD(\Omega)$ of test functions by 
\begin{equation}\label{eq-taudef}
	\tau_\lambda\psi =\psi\circ \rot_\lambda, \quad \psi\in \DD(\Omega).
\end{equation} 
Recall that a net $\{\phi_j\}$ converges in the space $\DD(\Omega)$,
 if each $\phi_j$ is supported in a fixed compact $K\subset \Omega$, and the net  $\{\phi_j\}$ converges in the Fréchet space $\DD_K, $ i.e. all partial derivatives converge uniformly on $K$.
 Using this it is easily verified that $\tau$ is a continuous representation of $\T^n$ in the space $\DD(\Omega)$.

Notice that  $\cx^n\setminus Z$ is a Reinhardt domain.  For a positive integer $k$ 
  we define the  norm $\norm{\psi}_k$ with respect to the polar coordinates for a function $\psi$  in $\DD(\cx^n\setminus Z)$ :
\begin{equation}
	\label{eq-knorm}
\norm{\psi}_k=\max_{\substack{\beta,\gamma\in \N^n\\ \abs{\beta}+\abs{\gamma}\leq k}}\,\sup_{\cx^n\setminus Z} \abs{\frac{\partial^\beta}{\partial r^\beta}\frac{\partial^\gamma}{\partial \theta^\gamma}\psi},\end{equation} 
where   the tuples $r\in (\rl^+)^n, \theta\in \rl^n$ are the polar coordinates on
$\cx^n\setminus Z$ specified by $z_j=re^{i\theta_j}$, and  $\frac{\partial^\beta}{\partial r^\beta}, \frac{\partial^\gamma}{\partial \theta^\gamma}$
are partial derivatives operators in the polar coordinates defined as in \eqref{eq-convention}.
From the formulas
\begin{equation}\label{eq-dxdy}
{	\frac{\partial }{\partial x_j}= \cos \theta_j \frac{\partial}{\partial r_j}- \frac{\sin \theta_j}{r_j}\frac{\partial }{\partial \theta_j} } \text{ and  }
{\frac{\partial }{\partial y_j}= {\sin \theta_j}\frac{\partial}{\partial r_j}+\frac{\cos \theta_j}{r_j}\frac{\partial }{\partial \theta_j}},
\end{equation}
we see that  for $0<\varrho_1<\varrho_2<\infty$, there is a constant $B_k(\varrho_1,\varrho_2)$ such that  for each compact set  $K$ such that
\[ K\subset \{ z\in \cx^n : \varrho_1< \abs{z_j}<\varrho_2\},\] 
we have
\begin{equation}
	\label{eq-equivalence}
	 \frac{1}{B_k(\varrho_1, \varrho_2)}\cdot\norm{\psi}_k \leq  \norm{\psi}_{\mathcal{C}^k} \leq B_k(\varrho_1,\varrho_2)\cdot\norm{\psi}_k.
\end{equation}
 We will need the following elementary
 estimate:
 \begin{prop}\label{prop-est} Let $\tau$ be the representation of $\T^n$ on $\cx^n\setminus Z$ given by \eqref{eq-taudef}. 
 	For integers $m,k\geq 0$, and a compact $K\subset \cx^n\setminus Z$ there is a constant $C>0$ such that for each $\psi\in \DD_K$ 
	 and each $\alpha\in \Z^n$, such that $\abs{\alpha_j}\geq 2k$ for $1\leq j \leq n$, we have 
	 
 	\[ \norm{\ppi^\tau_\alpha \psi}_{\mathcal{C}^k} \leq \frac{C}{\prod\limits_{j=1}^n\abs{\alpha_j}^m}\cdot \norm{\psi}_{\mathcal{C}^{2nk+nm}}.\] 
 \end{prop}

 \begin{proof} For $r\in(\rl^+)^n, \alpha\in \Z^n$, define the $\alpha$-th Fourier coefficient of $\psi$ by
 	\begin{equation}\label{eq-est0}
 		\wh{\psi}(r,\alpha)= \frac{1}{(2\pi)^n}\int\limits_{[0,2\pi]^n} \psi(re^{i\theta})e^{-i\ipr{\alpha,\theta}}d\theta, 
 	\end{equation}
 where $re^{i\theta}= (r_1e^{i\theta_1}, \dots , r_n e^{i\theta_n})\in \cx^n\setminus Z$, $\ipr{\alpha,\theta}=\sum_{j=1}^n \alpha_j\theta_j$ and 
 $d\theta=d\theta_1\dots d\theta_n$ is the Lebesgue measure. 
 	We will also write $(\psi)^\wedge$ for $\wh{\psi}$ whenever convenient. We note the following properties:
 	\begin{enumerate}[wide]
 		\item We clearly have
 		\begin{equation}\label{eq-supnorm}
 			\sup_{r\in(\rl^+)^n}\abs{\wh{\psi}(r,\alpha)} \leq \sup_{\cx^n\setminus Z}\abs{\psi}.
 		\end{equation} 
 		\item If $\alpha\not=0$, by a standard integration by parts argument, for each $\ell\in \N^n$,
 		\begin{equation}\label{eq-est1}
 			 \wh{\psi}(r,\alpha)= \frac{1}{(i\alpha)^\ell}\left(\frac{\partial^\ell \psi}{\partial \theta^\ell}\right)^\wedge(r,\alpha),
 		\end{equation}
 	where, as usual, we set for $\zeta\in \cx^n$ and $\beta\in \N^n$, $\zeta^\beta=\zeta_1^{\beta_1}\dots \zeta_n^{\beta_n}$.
 		\item Differentiating under the integral sign we have, for any  $\beta\in \N^n$,
 		\[
 		\left(\frac{\partial}{\partial r}\right)^\beta \wh{\psi}(r,\alpha)
 		= \frac{1}{(2\pi)^n}\int\limits_{[0.2\pi]^n} \frac{\partial^\beta\psi}{\partial r^\beta}(re^{i\theta})e^{i\ipr{\beta, \theta}} e^{-i\ipr{\alpha,\theta}}d\theta
 		={\left(\frac{\partial^\beta\psi}{\partial r^\beta}\right)}^\wedge(r,\alpha-\beta),
\]
 	and combining this with \eqref{eq-est1}, we see that if $\beta\not=\alpha$, we have for each integer $\ell\geq0 $  that
 	\begin{equation}\label{eq-est3}
 	\left(\frac{\partial}{\partial r}\right)^\beta \wh{\psi}(r,\alpha)= \frac{1}{\left(i(\alpha-\beta)\right)^\ell}\left(\frac{\partial^\ell}{\partial\theta^\ell}\frac{\partial^\beta}{\partial r^\beta}\psi\right)^\wedge(r,\alpha-\beta).
 	\end{equation}
 
	\end{enumerate}

Now   for $re^{i\theta}\in \cx^n\setminus Z$ the evaluation $\DD(\cx^n\setminus Z)\to \cx$, $\psi\mapsto \psi(re^{i\theta})$ is continuous, 
so using \eqref{eq-t} we see  that for each $\alpha\in \Z^n$ and each $\psi\in \DD(\cx^n\setminus Z)$, we have
 \begin{align}
 	\ppi^\tau_\alpha\psi(re^{i\theta})
 	&= \int_\T \lambda^{-\alpha}(\tau_\lambda\psi) (re^{i\theta}) d\lambda
 = \frac{1}{(2\pi)^n}\int_{[0,2\pi]^n} e^{-i\ipr{\alpha,\phi}}\psi(e^{i\phi} \cdot re^{i\theta})d\phi= e^{i\ipr{\alpha,\theta}}\cdot \wh{\psi}(r,\alpha),\label{eq-pi}
 \end{align}
 where in the last step we make a change of variables in the integral from $\phi$ to $\phi+\theta$.  Therefore, if $\beta, \gamma\in \N^n$ with $\beta_j\not=\alpha_j$ for $1\leq j \leq n$,  then
\[ \frac{\partial^\gamma}{\partial\theta^\gamma}\frac{\partial^\beta}{\partial r^\beta}(	\ppi^\tau_\alpha\psi)(re^{i\theta})=e^{i\ipr{\alpha,\theta}}\cdot  \frac{(i\alpha)^\gamma}{\left(i(\alpha-\beta)\right)^\ell}{\left(\frac{\partial^\ell}{\partial\theta^\ell}\frac{\partial^\beta}{\partial r^\beta}\psi\right)}^\wedge(r,\alpha-\beta).\]

Now in the above formula, if we have $\abs{\beta}+\abs{\gamma}\leq k$, and let 
\[ \ell= (k+m, \dots, k+m)= (k+m)\bm{1},\]
where $\bm{1}\in \N^n$ is the multi-index each of  whose $n$ entries is 1, we obtain, after taking absolute values of both sides
	\begin{equation}\label{eq-step1}
\abs{\frac{\partial^\gamma}{\partial\theta^\gamma}\frac{\partial^\beta}{\partial r^\beta}(	\ppi^\tau_\alpha\psi)(re^{i\theta})}
	=\frac{\prod\limits_{j=1}^n\abs{\alpha_j}^{\gamma_j}}{\prod\limits_{j=1}^n\abs{\alpha_j-\beta_j}^{m+k}}
	\cdot \abs{\left(\frac{\partial^{(m+k)\bm{1}}}{\partial\theta^{(m+k)\bm{1}}}\frac{\partial^\beta}{\partial r^\beta}\psi\right)^\wedge(r,\alpha-\beta)}.
\end{equation}
In the first factor on the right hand side, we  have by hypothesis for each $j$ that $\abs{\alpha_j}\geq 2k$ and $0\leq \beta_j,\gamma_j \leq k$. 
Therefore, we have $\abs{\alpha_j-\beta_j}\geq \frac{1}{2}\abs{\alpha_j}$. The first factor can be estimated as
\begin{equation}\label{eq-step2}
\frac{\prod\limits_{j=1}^n\abs{\alpha_j}^{\gamma_j}}{\prod\limits_{j=1}^n\abs{\alpha_j-\beta_j}^{m+k}}\leq \frac{\prod\limits_{j=1}^n\abs{\alpha_j}^{k}}{\prod\limits_{j=1}^n\left(\frac{1}{2}\abs{\alpha_j}^{m+k}\right)}=\frac{2^{m+k}}{\prod\limits_{j=1}^n\abs{\alpha_j}^{m}}. 
\end{equation}
Using \eqref{eq-supnorm}, the second factor can be estimated as
\begin{equation}\label{eq-step3}
\abs{\left(\frac{\partial^{(m+k)\bm{1}}}{\partial\theta^{(m+k)\bm{1}}}\frac{\partial^\beta}{\partial r^\beta}\psi\right)^\wedge(r,\alpha-\beta)}\leq \sup_{\cx^n\setminus Z} \abs{\frac{\partial^{(m+k)\bm{1}}}{\partial\theta^{(m+k)\bm{1}}}\frac{\partial^\beta}{\partial r^\beta}\psi}\leq \norm{\psi}_{2nk+nm} ,
\end{equation}
where in the last step we use the norm introduced in \eqref{eq-knorm}, and used the fact that 
\[ \abs{(m+k)\bm{1}+\beta}= (m+k)n + \abs{\beta}\leq (m+k)n +nk,\]
since each $\beta_j\leq k$.
 Combining \eqref{eq-step1}, \eqref{eq-step2} and \eqref{eq-step3} we see that
\[\abs{\frac{\partial^\gamma}{\partial\theta^\gamma}\frac{\partial^\beta}{\partial r^\beta}(	\ppi^\tau_\alpha\psi)(re^{i\theta})}\leq \frac{2^{m+k}}{\prod\limits_{j=1}^n\abs{\alpha_j}^{m}}\cdot  \norm{\psi}_{2nk+nm},\]
from which, taking a supremum on the left hand side over $re^{i\theta}\in\cx\setminus Z$, and remembering that $\abs{\beta}+\abs{\gamma}\leq k$,
 we conclude that 
 \[ \norm{\ppi^\tau_\alpha\psi}_k\leq \frac{2^{m+k}}{\prod\limits_{j=1}^n\abs{\alpha_j}^{m}}\cdot  \norm{\psi}_{2nk+nm}.\]
 Recall that $K$ is a compact set in $\cx^n\setminus Z$ such that $\psi\in \DD_K$, i.e. , the support of $\psi$ is contained in $K$. Suppose that
 $\varrho_1, \varrho_2$ are such that $K\subset A$, where $A$ is the product of annuli $A=\{ z\in \cx^n: \varrho_1<\abs{z_j}< \varrho_2, 1\leq j \leq n \}$. Formulas \eqref{eq-est0} and \eqref{eq-pi} show that the compact support of $\ppi^\tau_\alpha\psi$ is also contained in the set $A$. Therefore,
 passing to the equivalent $\mathcal{C}^k$-norms using \eqref{eq-equivalence}, we have
 \[  \norm{\ppi^\tau_\alpha\psi}_{\mathcal{C}^k}\leq  B_k(\varrho_1, \varrho_2)\cdot B_{2nk+nm}(\varrho_1, \varrho_2)\cdot \frac{2^{m+k}}{\prod\limits_{j=1}^n\abs{\alpha_j}^{m}}\cdot \norm{\psi}_{\mathcal{C}^{2nk+nm}},\]
which completes the proof of the result.
 \end{proof}
\subsection{The dual representation} Let $\Omega$ be a Reinhardt domain.
Since for each $\lambda\in \T^n$, the map $\rot_\lambda$ maps  $\Omega$ biholomorphically (and therefore diffeomorphically) to itself. Consequently, we can define a representation of $\T^n$ on the space of distributions $\DD'(\Omega)$ using the pullback operation
\begin{equation}\label{eq-sigmadef}\sigma_\lambda (T)= (\rot_\lambda)^*(T), \quad T\in \DD'(\Omega).
\end{equation}
The representation $\sigma_\lambda$ is closely related to the representation $\tau_\lambda$ introduced in \eqref{eq-taudef}.
Clearly, $\DD(\Omega)$ is an invariant (dense)  subspace of $\sigma$,  on which  $\sigma$ restricts to $\tau$.
So $\sigma$ is simply the extension of the representation $\tau$ of \eqref{eq-taudef} by continuity to the space of distributions. 

The representation $\sigma$ on $\DD'(\Omega)$ is  also  ``dual" to the representation $\tau$ on $\DD(\Omega)$:
\begin{equation}\label{eq-dual}
	 \ipr{\sigma_\lambda(T), \phi}=\ipr{ (\rot_\lambda)^*(T),\phi}= \ipr{T, \frac{1}{\det_\rl (\rot_\lambda)}\phi\circ \rot_\lambda^{-1}}= \ipr{T, \tau_{\lambda^{-1}}\phi},
\end{equation}
so that $\sigma_\lambda$ is the transpose of the map $\tau_{\lambda^{-1}}$. The second equality in the chain may be proved by change of variables when $T$ is a test function, and then using density. 

\begin{prop}\label{prop-sigmacont}
	The representation $\sigma$ of $\T^n$ on $\DD'(\Omega)$ defined in \eqref{eq-sigmadef} is continuous.
\end{prop}
\begin{proof} Let $\{\lambda_j\}$ be a sequence in $\T^n$ converging to the identity element.  Since pointwise convergence of a sequence in $\DD'(\Omega)$ on each test function implies convergence in the strong dual topology (see \cite{treves}), to show that $\sigma_{\lambda_j}(T)\to T$ in 
	$\DD'(\Omega)$ we need to show that for $\phi\in \DD(\Omega)$ we have $\ipr{\sigma_{\lambda_j}(T),\phi}\to \ipr{T,\phi}$. But by \eqref{eq-dual}, 
	$\ipr{\sigma_{\lambda_j}(T),\phi}=\ipr{T, \tau_{\lambda_{j}^{-1}}\phi}$ and by the continuity of the representation $\tau$, to follows that 
	 $\sigma_{\lambda_j}(T)\to T$ in 
	$\DD'(\Omega)$. 
	
	To complete the proof, by Proposition~\ref{prop-continuity}, we need to show that the topology of $\Omega$ is generated by a collection of 
$\sigma$-invariant seminorms. Let $B$ be a bounded subset of $\DD(\Omega)$. Let 
\[ \wt{B}= \{ \sigma_\lambda(\phi): \lambda\in \T^n, \phi\in B\}.\]
Then it is clear from \eqref{eq-bounded} that $\wt{B}$ is also a bounded subset of $\DD(\Omega)$. With notation as it \eqref{eq-seminorm},
it is clear that for each $T\in \DD'(\Omega)$, we have  $p_B(T)\leq p_{\wt{B}}(T)$.  Therefore the topology of $\DD'(\Omega)$ is generated by 
the family of continuous invariant seminorms $\{ p_{\wt{B}}\}$.
	\end{proof}

\subsection{Fourier series of distributions on Reinhardt domains} Let $\Omega$ be a Reinhardt domain and let $T\in \DD'(\Omega)$ be a 
distribution. Then by the results of  Section~\ref{sec-abstract} we can expand $T$ in a formal Fourier series with respect to the representation 
$\sigma$ of \eqref{eq-sigmadef}. For simplicity of notation, whenever there is no possibility of confusion, we denote the Fourier components of
$T$ by
\begin{equation}\label{eq-formalfourier-t}
 T_\alpha=\ppi^\sigma_\alpha(T),\quad  \alpha\in \Z^n,\end{equation}
so that the Fourier series of $T$ is written as $T\sim \sum\limits_{\alpha\in \Z^n} T_\alpha$. We notice the following properties of the Fourier components:
\begin{prop}
	\label{prop-coefficients} Let $\Omega$ and $T$ be as above and let $\alpha\in \Z^n$. 
	\begin{enumerate}[label=(\alph*)]
		\item If the distribution $T$ lies in one of the linear subspaces $\OO(\Omega), \mathcal{C}^\infty(\Omega), (\OO\cap \mathcal{C}^\infty)(\Omega)$ or
		$\DD(\Omega)$ of $\DD'(\Omega)$, the Fourier component $T_\alpha$ also lies in the same subspace.
		\item If $U$ is another Reinhardt domain such that $U\subset \Omega$, we have
		\begin{equation}\label{eq-restriction}
			 T_\alpha|_U=\all{T|_U}_\alpha.
		\end{equation}
	\end{enumerate}
\end{prop}
\begin{proof}All these are consequences of Part 2 of Proposition~\ref{prop-fourierproperties}. To see  Part (a), let $X$ be the space $\DD'(\Omega)$ and 
	let $Y$ be one of the spaces $\OO(\Omega), \mathcal{C}^\infty(\Omega)$ or $\DD(\Omega)$. Each of these has a natural structure
	of a complete LCTVS (the space $\OO(\Omega)$ as  a closed subspace of $\DD'(\Omega)$, $\mathcal{C}^\infty(\Omega)$ in its Fréchet topology and 
$\DD(\Omega)$ in its $LF$-topology), and for each one of these topologies, the representation $\sigma$ restricts to a continuous representation
(in the stronger topology of the subspace), which we will call $\tau$ (this coincides with the $\tau$ introduced in \eqref{eq-taudef}). If $j$ is the inclusion map of any one of 
these subspaces into $\DD'(\Omega)$, then clearly it is continuous (where the subspace has the natural topology). Therefore, by 	
\eqref{eq-equivariance}, we have $j\left(\ppi^\tau_\alpha T\right)= \ppi^\sigma_\alpha (j(T))$. By definition, the right hand side is precisely $T_\alpha$, i.e. the Fourier component of $T$ as an element of $\DD'(\Omega)$. Notice that $\ppi^\tau_\alpha T\in Y$, since it is the 
Fourier component of $T$ as an element of $Y$. Since $j$ is the inclusion of $Y$ in $X$, it now follows that $T_\alpha\in Y$ as well.

Since the result is true for $\OO(\Omega)$ and $\mathcal{C}^\infty(\Omega)$ it follows for $(\OO\cap\mathcal{C}^\infty )(\Omega)$.

For part (b), in part 2 of Proposition~\ref{prop-fourierproperties}, let $X=\DD'(U)$ with representation defined in \eqref{eq-sigmadef}, which we call
$\sigma'$ for clarity, and let $Y=\DD'(\Omega)$ with the representation $\sigma$ and let $j:Y\to X$ be the restriction map of distributions, which 
is clearly continuous and intertwines the two representations. Then \eqref{eq-equivariance} takes the form $j\circ \ppi^\sigma_\alpha=\ppi^{\sigma'}_\alpha \circ j$. Applying this to the distribution $T$, we see that
\[ T_\alpha|_U = j\all{\ppi^\sigma_\alpha T}=\ppi^{\sigma'}_\alpha\all{j(T)}= \all{T|_U}_\alpha. \]

\end{proof}

We  establish the following lemma:
\begin{lem}
	\label{lem-weierstrass}
	Let $\Omega\subset \cx^n $ be a Reinhardt domain, and let 
	\[ \mathscr{A}= \{T\in \OO(\Omega): \text{each Fourier component  } T_\alpha \text{ belongs to } \mathcal{C}^\infty (\Omega)\}.\]
	Suppose that for each $T\in \mathscr{A}$  the Laurent series $\sum\limits_{\alpha\in \Z^n}T_\alpha$  converges absolutely in the Fréchet space
	 $\mathcal{C}(\Omega)$ .
	Then  the Laurent series  of each element of $\mathscr{A}$  converges absolutely in the space $\mathcal{C}^\infty(\Omega)$.
\end{lem}	
\begin{proof} 
	  For a  function $f\in \mathcal{C}(\Omega)$ and a compact set $K$, let $p_K$ be the seminorm:
	\begin{equation}\label{eq-pk}
		p_K(f)=\sup_{K}\abs{f}. 
	\end{equation}
	Notice that the family $\{p_K: K \subset \Omega \quad \text{compact}\}$ is a generating family of seminorms for $\mathcal{C}(\Omega)$, and the hypothesis on $\mathscr{A}$  can be expressed as
			\begin{equation}\label{eq-hypo1}
			\sum\limits_{\alpha\in \Z^n} p_K(T_\alpha)<\infty,
		\end{equation}
		for each compact $K\subset \Omega$.

 The Fréchet topology of $\mathcal{C}^\infty(\Omega)$ is generated by the seminorms $p_{K,L}$, where $K$ ranges over compact subsets of $\Omega$,
$L$ is a constant coefficient differential operator on $\cx^n$,  and
\[ p_{K,L}(f)= \sup_{K}\abs{Lf}.\]
We will continue to  write $p_K$ for the seminorm \eqref{eq-pk}  corresponding to $L=1$.  From the distributional chain rule \eqref{eq-chain1}, we have for each $j$ that
\[ \dd_j\sigma_\lambda(T)= \dd _j(\rot_\lambda^* T) =(\dd_j \rot_\lambda)\cdot \rot_\lambda^* (\dd_j T)= \lambda_j\cdot \sigma_\lambda (\dd_j T).\]
Let $\beta\in \N^n$, and denote  $\dd^\beta=\dd_1^{\beta_1}\dots \dd_n^{\beta_n}$. Then, from the above we have
\begin{align*}
	\dd^\beta T_\alpha= \dd^\beta \int_{\T^n} \lambda^{-\alpha}\sigma_\lambda(T) d\lambda = \int_{\T^n} \lambda^{-\alpha} \dd^\beta \sigma_\lambda(T) d\lambda =\int_{\T^n} \lambda^{-(\alpha-\beta)}\sigma_\lambda(\dd^\beta T)d\lambda=(\dd^\beta T)_{\alpha-\beta},
\end{align*}
where $(\dd^\beta T)_{\alpha-\beta}=\ppi^\sigma_{{\alpha-\beta}}(\dd^\beta T)$. It follows that all Fourier coefficients of the distribution $\dd^\beta T$ are in $\mathcal{C}^\infty(\Omega)$, so $\dd^\beta T\in \mathscr{A}$.
Therefore,  
\[ p_{K,\dd^\beta}(T_\alpha)= p_K\left(\dd^\beta T_\alpha\right)=p_{K}\left((\dd^\beta T)_{\alpha-\beta}\right).\]
It follows therefore that 
\[ \sum_{\alpha\in \Z^n} p_{K,\dd^\beta}(T_\alpha)= \sum_{\gamma\in \Z^n}p_{K}\left((\dd^\beta T)_{\gamma}\right)<\infty,\]
using assumption \eqref{eq-hypo1} on elements of $\mathscr{A}$. 
 A partial derivative $L=\frac{\partial^\gamma}{\partial x^\gamma}\frac{\partial^\delta}{\partial y^\delta}$ on  $\cx^n$ can be rewritten as a polynomial in the commuting differential operators $\dd_1,\dots, \dd_n,
\db_1, \dots, \db_n$, and by hypothesis the derivatives $\db_j$ are zero. Therefore, it follows that $\sum_\alpha p_{K,L}(T_\alpha)<\infty$. Consequently the series $\sum_\alpha T_\alpha$ is absolutely convergent in
$\mathcal{C}^\infty(\Omega)$.

\end{proof}

\subsection{ Proof  of Theorem~\ref{thm-distn}: case of annular domains} We now prove a weaker version of 
 Theorem~\ref{thm-distn}.
A Reinhardt domain $\Omega\subset \cx^n$ will be called \emph{annular} if it is disjoint from  the coordinate hyperplanes:
\begin{equation}\label{eq-annular}
	\Omega\cap Z=\emptyset, 
\end{equation}
where $Z$ is as in \eqref{eq-Z}.  We have the following:
\begin{prop}\label{prop-annular}
	Let $\Omega \subset \cx^n\setminus Z$ be an annular Reinhardt domain in $\cx^n$.
	Then for each $\alpha\in \Z^n$ there is a continuous linear functional 
	$a_\alpha:\OO(\Omega)\to \cx$ such that for each $T\in \OO(\Omega)$, the series of terms in $\mathcal{ C}^\infty (\Omega)$ given by
$
		\sum_{\alpha\in \Z^n} a_\alpha(T)e_\alpha
$
	converges absolutely in $\mathcal{C}^\infty({\Omega})$ to a function $f$, and  this $\mathcal{C}^\infty$-smooth function  generates 
	the distribution $T$.
\end{prop}
\begin{proof}
Let $T\in \OO(\Omega)$ and let $T_\alpha$ be its $\alpha$-th Fourier component as in \eqref{eq-formalfourier-t},
 and let  $S=e_{-\alpha}\cdot T_\alpha$. The distributional Leibniz rule
 \begin{equation}\label{eq-leibniz}
 	 \db_j(fU)= \left(\db_j f \right)\cdot U+f \cdot \left(\db_j U\right), \quad f\in \mathcal{C}^\infty(\Omega)  \text{ and  } U\in \DD'(\Omega),
 \end{equation}
shows that $S\in \OO(\Omega)$,  since $T_\alpha\in \OO(\Omega)$ by Part (1) of Proposition~\ref{prop-coefficients}.
By Part (1) of Proposition~\ref{prop-fourierproperties},  the Fourier component $T_\alpha$ lies in the $\alpha$-th Fourier mode of 
$\OO(\Omega)$, i.e.,
$\sigma_\lambda(T_\alpha)=\lambda^\alpha\cdot T_\alpha$, so
 \[ \sigma_\lambda(S)= \lambda^{-\alpha}e_{-\alpha}\cdot \lambda^\alpha T_\alpha=S.\]
Therefore, using polar coordinates $z_j=r_je^{i\theta_j}$, with $\epsilon_j(h)=(1,\dots, 1, e^{ih}, \dots, 1)\in \T^n$ (with $e^{ih}$ in the $j$-th position) we have
\[ \frac{\partial}{\partial \theta_j}S=\lim_{h\to 0}\frac{1}{h}\left(\sigma_{\epsilon_j(h)}(S)-S\right)=0.\]
Since $\db_j S=0$,  and in polar coordinates we have
\begin{equation}
	\label{eq-dbarpolar}
	\db_j= \frac{e^{i\theta_j}}{2}\left(\frac{\partial}{\partial r_j} + \frac{i}{r_j}\frac{\partial}{\partial \theta_j} \right),
\end{equation}
 we see that $$\dfrac{\partial}{\partial r_j}S=0$$ also. Now  from the representations \eqref{eq-dxdy}
 we have that
$\displaystyle{  \frac{\partial}{\partial x_j} S= 0}$  and $ \displaystyle{\frac{\partial}{\partial y_j} S= 0}$  in $ \DD'(\Omega)$ for $1\leq j\leq n$.
Therefore, by a classical result in the  theory of distributions (see \cite[Theorème VI, pp. 69ff.]{schwartz}) we have that  $S$ is locally  constant on $\Omega$, or more precisely, $S$ is  generated by a locally constant function. Since $\Omega$ is connected, it follows that there
is a constant $a_\alpha(T)\in \cx$  which generates the distribution $S$, so that $ T_\alpha= a_\alpha(T)e_\alpha.$
This can also be written as $a_\alpha(T)=e_{-\alpha}\cdot T_\alpha= e_{-\alpha}\cdot \ppi^\sigma_\alpha (T)$.  Since multiplication by the fixed $\mathcal{C}^\infty$-function
 $e_{-\alpha}$  is a 
continuous map on  $\DD'(\Omega)$ (and therefore on the  subspace $\OO(\Omega)$), and $\ppi^\sigma_\alpha: \OO(\Omega)\to \OO(\Omega)$ is continuous by part 1 of Proposition~\ref{prop-fourierproperties}, it follows that $a_\alpha:\OO(\Omega)\to \OO(\Omega)$ is continuous. But we know that $a_\alpha$ takes values in the subspace of distributions 
generated by constants. By a well-known theorem, on a finite-dimensional topological vector space, there is only one  
topology.  Therefore the topology induced from $\OO(\Omega)$ on the subspace of constants coincides with the natural topology of $\cx$.  Therefore 
$a_\alpha:\OO(\Omega)\to \cx$ is continuous.

By the above,  the Fourier components $T_\alpha$ of the Fourier series of $T$  actually
  lie in   $(\OO\cap\mathcal{C}^\infty)(\Omega)$, and thus any partial sum (i.e. sum of finitely many terms) of this series also lies in  $(\OO\cap\mathcal{C}^\infty)(\Omega)$.
  We will now show  that the series \eqref{eq-formalfourier-t}  is absolutely convergent in the topology of  $\mathcal{C}^\infty(\Omega)$, 
  in the sense of Proposition~\ref{prop-absolute}.
 Let $K\subset \Omega$ be compact.  Let $\delta< \mathrm{dist}(K, \cx^n\setminus \Omega)$, and 
let $\Psi\in \DD(\cx^n)$ be a   test function such that
	\[ \mathrm{support}(\Psi)\subset B(0,\delta),  \quad  \int_\cx \Psi(\zeta)dV(\zeta)=1, \quad \text{ and } \Psi(\zeta)= \Psi (\abs{\zeta_1},\dots, \abs{\zeta_n}) \text{ for  $\zeta\in \cx^n$},\]
	i.e. $\Psi$ is radial around the origin in each complex coordinate.
For $z\in K$, define $\psi_z(\zeta)=\Psi(\zeta-z),$
so that $\psi_z$ is radially symmetric around $z$ in each complex direction. Therefore, for $z\in K$, we have, by Lemma~\ref{lem-radial} for each $\alpha$ that
\[ \ipr{T_\alpha,\psi_z}=  \ipr{a_\alpha(T)e_\alpha, \psi_z}=a_\alpha(T)z^\alpha.\]

 By the continuity of  the mapping $T:\DD({\Omega})\to \cx$, there is a constant $C>0$ and an integer $k\geq 0$  such that for all $\psi\in \DD({\Omega})$ with support in $K$, we have that 
 \begin{equation}\label{eq-distn1}
 	\abs{\ipr{T,\psi}}\leq C\cdot\norm{\psi}_{\mathcal{C}^k({\Omega})}.
 \end{equation}
Also notice that for any $\psi\in \DD({\Omega})$ we have 
 	\begin{align}
 	\ipr{\ppi^{\sigma}_\alpha(T),\psi} &= \ipr{\int_{\T^n} \lambda^{-\alpha}\sigma_\lambda(T)d\lambda, \psi}= \int_{\T^n} \lambda^{-\alpha}\ipr{\sigma_\lambda(T),\psi}d\lambda \qquad \text{using \eqref{eq-t}}\nonumber\\
 	&= \int_{\T^n} \lambda^{-\alpha}\ipr{T,\tau_{\lambda^{-1}}\psi}d\lambda = \ipr{T,  \int_{\T^n} \lambda^{-\alpha}\tau_{\lambda^{-1}}\psi d\lambda}\qquad \text{using \eqref{eq-dual} and \eqref{eq-t}}\nonumber\\
 	&= \ipr{T,  \int_{\T^n} \lambda^{\alpha}\tau_{\lambda}\psi d\lambda}\nonumber\\&\text{ using the invariance of Haar measure of a compact group under inversion}\nonumber\\
 	&=\ipr{T, \ppi^\tau_{-\alpha}\psi}.
 	\label{eq-dualexp}
 \end{align}

In the following estimates, let $C$ denote a constant that depends only on the compact $K$ and the distribution $T$, and may have different values at different occurrences.  By combining \eqref{eq-distn1} with \eqref{eq-dualexp}, we see that for  $z\in K, \phi\in \DD({\Omega})$  and each $\alpha$   with $\abs{\alpha}\geq 2k$ ,  using Lemma~\ref{lem-radial}
 \begin{align*}
 	\abs{a_\alpha(T)z^\alpha}&=\abs{\ipr{T_\alpha, \psi_z}}=\abs{\ipr{\ppi^\sigma_\alpha T,\psi_z}}=\abs{\ipr{T, \ppi^\tau_{-\alpha}\psi_z }}\nonumber\\
 	&\leq C\cdot\norm{\ppi^\tau_{-\alpha}\psi_z}_{\mathcal{C}^k({\Omega})}\quad \text{using \eqref{eq-distn1}}\\
 	&\leq  \frac{C}{\abs{\alpha}^2}\cdot \norm{\psi_z}_{\mathcal{C}^{2nk+2n}}\quad \text{by Proposition~\ref{prop-est} with $m=2$,}\\
 	&=\frac{ C}{\abs{\alpha}^2}\cdot \norm{\psi_z}_{\mathcal{C}^{2nk+2n}},
 \end{align*}
recalling that the local order $k$ depends only on the distribution $T$ and the compact set $K$. Now,
	for each $z$  we have $\norm{\psi_z}_{\mathcal{C}^{2nk+2n}}= \norm{\Psi}_{\mathcal{C}^{2nk+2n}}$ by translation invariance of the
	 norm. Therefore, for each $\alpha\geq 2k$  we obtain the estimate for the seminorm (with the same convention as above on the constant $C$)
	 \begin{equation}
	     \label{eq-pkconv}
	     p_K(T_\alpha)=p_{K}(\ppi^\sigma_\alpha T)=\sup_{z\in K}	\abs{a_\alpha(T)z^\alpha} \leq  \frac{ C}{\abs{\alpha}^2}\cdot \norm{\Psi}_{\mathcal{C}^{2nk+2n}}= \frac{ C}{\abs{\alpha}^2} .
	 \end{equation}
Clearly therefore
 $\sum\limits_{\alpha\in 	\Z^n} p_K(T_\alpha)<\infty$. By Lemma~\ref{lem-weierstrass}, the series $\sum_\alpha T_\alpha$ converges absolutely in  $\mathcal{C}^\infty(\Omega)$.
 Let $f$ be its sum.
Since the inclusion $\mathcal{C}^\infty(\Omega)\subset \DD'(\Omega)$ is 
continuous, we see easily that the Fourier series $\sum_\alpha T_\alpha$ of $T\in \DD'(\Omega)$ 
converges absolutely in $\DD'(\Omega)$. Now it follows from Corollary \ref{cor-cesaro} that the sum of the series is $T$. Thus $T$ is the distribution generated by $f$, and this completes the proof of the proposition.
\end{proof}

\subsection{Smoothness of holomorphic distributions}
 \begin{cor}
 	\label{cor-smooth}
 	Let $\Omega\subset \cx^n$ be an open set. Then each holomorphic distribution on $\Omega$ is generated by 
 	a $\mathcal{C}^\infty$-smooth function, i.e. $\OO(\Omega)=(\OO\cap \mathcal{C}^\infty)(\Omega)$.
 	\end{cor}
 \begin{proof} Let $T\in \OO(\Omega)$. It suffices to show that  $T$ is a smooth function in a neighborhood of each point $p\in \Omega$. Without loss of generality,
 	thanks to the invariance of the space of holomorphic functions under translations (Proposition~\ref{prop-invariance} and comments after it) , we can assume that 
 	$p=0$. Let $r>0$ be such that the polydisc $P(r)$ given by 
 	 	 	\[ P(r)=\{\abs{z_j}<r, 1\leq j \leq n\}\subset \cx^n\]
 	 	 	is contained in $\Omega$. We will show that $T$ is a smooth function on $P'= P(\frac{r}{6})$, i.e., $T\in (\OO\cap\mathcal{ C}^\infty)(P')$.

 	 By the previous section,  holomorphic distributions on an annular Reinhardt domain are smooth. By translation invariance, the 
 	 same is true if the annular domain is centered at a point $a\in \cx^n$ different from  the origin, i.e. for a domain of the form 
 	$ \tra_a(A)$
 	where $A$ is an annular Reinhardt domain centered at the origin, with $\tra_a$ as in \eqref{eq-tra} is the translation by $a$.

 	  Indeed, let $a=\left(\frac{r}{3},\dots, \frac{r}{3}\right)$, and set
 	  $Q=\tra_a\all{P(\frac{r}{2})\setminus Z}$, so that   $Q$ is the annular Reinhardt domain centered at $a$  given by
 	\[ Q=\left\{ z\in \cx^n: 0<\abs{z_j-\frac{r}{3}}< \frac{r}{2}, 1\leq j \leq n\right\},\]
 	for which it is easy to verify that $P'\subset Q\subset P(r)$. Now since $T|_{Q}\in (\OO\cap \mathcal{ C}^\infty)(Q)$ is a holomorphic function, the result follows.
 \end{proof}

\subsection{Extension to the log-convex hull} In this section, we prove the following special case of
Theorem~\ref{thm-distn} for annular Reinhardt domains:
\begin{prop}
        \label{prop-logconvex}
        Let $\Omega$ be an annular Reinhardt domain in $\cx^n$. Then the Laurent series of a distribution $T\in \OO(\Omega)$ 
        converges absolutely in $\mathcal{C}^\infty(\wh{\Omega})$ where $\wh{\Omega}$ is the smallest log-convex annular Reinhardt domain containing $\Omega$.
\end{prop}
Introduce the following notation:
for a compact subset $K\subset \cx^n\setminus Z= (\cx^*)^n$, we let
\begin{equation}\label{eq-khat}
	\wh{K}= \left\{z\in \cx^n\setminus Z: \abs{e_\alpha(z)}\leq \sup_K\abs{e_\alpha}, \text{for all $\alpha\in \Z^n$}\right\}.
\end{equation}
\begin{lem}\label{lem-logconvex}
	Let $\Omega$ be an annular Reinhardt domain in $\cx^n$ and let $\wh{\Omega}$ be the log-convex annular Reinhardt domain containing $\Omega$. Then for each point $p\in \wh{\Omega}$ 
	there is a compact neighborhood $M$ of $p$ in $\wh{\Omega}$ and a compact subset $K\subset \Omega$ such that $M\subset \wh{K}$.  
\end{lem}

\begin{proof}
    Let $\Lambda:\cx^n\setminus Z\to  \rl^n$ be the map
\begin{equation}\label{eq-Lambda}
	\Lambda(z_1,\dots, z_n)= (\log\abs{z_1},\dots, \log \abs{z_n}).\end{equation}
	From \eqref{eq-logconvex1}, it follows that an annular Reinhardt domain $ V\subset \cx^n$ is  {log-convex} if and 
	only if
the subset $\Lambda(V)$ of $\rl^n$ (called the \emph{ logarithmic shadow } of $V$) is convex. 
The set $\wh{\Omega}$ is characterized by the fact that it  is the
 convex hull of $\Lambda(\Omega)$.
Therefore, if $p\in \wh{\Omega}$, there exist $m$ points $q^1, \dots, q^m\in V\setminus Z$ and positive numbers
 $t_1,\dots, t_m$, with $ \sum\limits_{k=1}^m t_k =1$ such that for $1\leq j \leq n$, we have $\Lambda(p)= \sum_{k=1}^m t_k \Lambda(q^k)$, i.e.,
\[ \log \abs{p_j}=\sum\limits_{k-1}^{m} t_k \log \abs{q_j^k}.\]
In other words $\Lambda(p)$ lies in the convex hull of the $m$-element set $\{\Lambda(q^1),\dots, \Lambda(q^m)\}$.
(By a theorem of Carathéodory, $m\leq n+1$, but we do not need this fact.) 
We will show that we can ``fatten" one the points in this set so that the convex hull of the fattened set contains a neighborhood of the point $\Lambda(p)$.

Without loss of generality, we can assume that $t_m\not=0$. Consider the affine self-map  $F$ of $\rl^n$ given by 
\[ F(x)= t_mx + \sum_{k=1}^{m-1} t_k \Lambda (q_j^k),\]
which is clearly an affine automorphism of $\rl^n$.  Therefore if $L$ is a compact neighborhood of $\Lambda(p)$ in $\rl^n$,  then $F^{-1}(L) $ is a compact  neighborhood of  $\Lambda(q^m)$, which after shrinking can be taken to be contained in $\Lambda(\Omega).$
  Let $K_0=\Lambda^{-1}(F^{-1}(L))$, and set 
\[ K=  \{q^1,\dots, q^{m-1}\} \cup K_0.\]
We claim that $\wh{K}$ contains the compact neighborhood $\Lambda^{-1}(L)$ of the point $p$.
Let $\zeta\in \Lambda^{-1}(L)$. Then there is 
a point $z\in K_0$ such that  $\Lambda(\zeta)= \sum_{k=1}^{m-1} t_k \Lambda(q^k)+t_m \Lambda(z)$, i.e. for each $1\leq j \leq n$ we have
\[ \log\abs{\zeta_j}=\sum_{k=1}^{m-1} t_k \log \abs{q^k_j}+ t_m \log \abs{z_j}.\]
 For simplicity of writing, introduce new notation as follows: we let $z^k=q^k$  for $1\leq k \leq m-1$ and $ z^m= z$. 
Then for a
multi-index $\alpha\in \Z^n$ and the point $\zeta\in \Lambda^{-1}(L)$ we have
\begin{align}
	\abs{e_\alpha(\zeta)}&= \prod_{j=1}^{n} \abs{\zeta_j}^{\alpha_j}= \exp\all{\sum_{j=1}^n\alpha_j \log \abs{\zeta_j}}\nonumber\\
	&=  \exp\all{\sum_{j=1}^n\alpha_j \sum_{k=1}^{m} t_k \log\abs{z^k_j} }\nonumber= \exp\all{\sum_{k=1}^m t_k \sum_{j=1}^n\alpha_j  \log\abs{z^k_j}}\\
	&\leq \exp \all{\max_{1\leq k \leq m}\sum_{j=1}^n\alpha_j  \log\abs{z^k_j}} \label{eq-convex}\\
	&=\max_{1\leq k \leq m} \exp\all{\sum_{j=1}^n\alpha_j  \log\abs{z^k_j}}\nonumber\\
	&= \max_{1\leq k \leq m} \abs{e_\alpha(z^k) }= \max\left\{\max_{1\leq k \leq m-1}\abs{e_\alpha(q^k)}, \abs{e_\alpha(z)}\right\}  \nonumber\\
	&\leq \max_{K}\abs{e_\alpha},
 \end{align}
 this shows that $\zeta\in \wh{K}$, and completes the proof.
\end{proof}

\begin{proof}[Proof of Proposition~\ref{prop-logconvex}]
By Lemma~\ref{lem-weierstrass}
it is sufficient to show that the Laurent series $\sum_\alpha a_\alpha(T)e_\alpha$ (whose existence was 
proved in Proposition~\ref{prop-annular}, and whose terms lie in $(\OO\cap \mathcal{C}^\infty)(\Omega)$) converges absolutely 
in $\mathcal{C}(\wh{\Omega})$.  It is sufficient to show that each point $p$ of $\wh{\Omega}$ has a compact 
neighborhood $M$ such that we have 
\[ \sum_\alpha p_M(a_\alpha(T)e_\alpha) <\infty.\]
Now by Lemma~\ref{lem-logconvex}, there is a compact subset $K\subset \Omega$ such that $M\subset \wh{K}$. We observe, using the definition \eqref{eq-khat} that
\begin{align*}
    p_M(a_\alpha(T)e_\alpha)&\leq p_{\wh{K}}(a_\alpha(T)e_\alpha)= \sup_{\wh{K}}\abs{a_\alpha(T)e_\alpha} 
    = \abs{a_\alpha(T)}\cdot \sup_{\wh{K}}\abs{e_\alpha}\\
    & =\abs{a_\alpha(T)}\cdot \sup_{{K}}\abs{e_\alpha}=\sup_{{K}}\abs{a_\alpha(T)e_\alpha}=p_K( a_\alpha(T)e_\alpha),
\end{align*}
so that 
\[ \sum_\alpha p_M(a_\alpha(T)e_\alpha)\leq \sum_\alpha p_K(a_\alpha(T)e_\alpha)<\infty,\]
using the estimate \eqref{eq-pkconv} which holds since $K\subset \Omega$. 
\end{proof}

 \subsection{Laurent series on non-annular Reinhardt domains}We now consider the case of a general (i.e. possibly non-annular) Reinhardt domain $\Omega$. We begin by showing that the only monomials that occur in the 
 Laurent series are the ones smooth on $\Omega:$
\begin{prop}\label{prop-nonannular}
Let $\Omega$ be a (possibly non-annular) Reinhardt domain in $\cx^n$. Then for each $\alpha \in \mathcal{S}(\Omega)$ (where $\mathcal{S}(\Omega)$ is as in \eqref{eq-somega}), there is a continuous linear functional $a_{\alpha} : \OO(\Omega) \rightarrow \cx$ such that the Fourier series of a $T \in \OO(\Omega)$ is of the form 
\begin{equation}\label{eq-nonannular}
T \sim \sum_{\alpha \in \mathcal{S}(\Omega)} a_{\alpha}(T) e_{\alpha}.    
\end{equation}
\end{prop}
\begin{proof}
  Since $T\in \OO(\Omega)$ we have  $T|_{\Omega\setminus Z}\in \OO(\Omega \setminus Z)$.   Notice that $\Omega \setminus Z$ is 
an annular Reinhardt domain and therefore the proof of Proposition~\ref{prop-annular} shows that the Fourier components 
are given for $\alpha\in \Z^n$ by 
\[ \all{T|_{\Omega\setminus Z}}_\alpha =\wh{a}_\alpha(T|_{\Omega\setminus Z}) e_\alpha,\]
where $\wh{a}_\alpha:\OO(\Omega \setminus Z)\to \cx$ is the $\alpha$-th coefficient functional associated to the domain $\Omega\setminus Z$ (see Proposition~\ref{prop-annular}). Thanks  to \eqref{eq-restriction}, we know that $ \all{T|_{\Omega\setminus Z}}_\alpha= (T_\alpha)|_{\Omega\setminus Z}$, so we have
\[(T_\alpha)|_{\Omega\setminus Z}= \wh{a}_\alpha(T|_{\Omega\setminus Z}) e_\alpha.\]
By Proposition~\ref{prop-coefficients}, the Fourier components of a holomorphic distribution are holomorphic distributions, so 
we have $T_\alpha\in \OO(\Omega)$ since $T\in \OO(\Omega)$. Therefore, since by Corollary~\ref{cor-smooth},  each holomorphic 
distribution is a smooth function,  we know that $T_\alpha\in \mathcal{C}^\infty (\Omega)$.
 Therefore,
the function $T_\alpha|_{\Omega\setminus Z}=\wh{a}_\alpha(T|_{\Omega\setminus Z}) e_\alpha$ admits a $\mathcal{C}^\infty$ extension through $Z$. If $\wh{a}_\alpha(T|_{\Omega\setminus Z})\not=0$ for some $T\in \OO(\Omega)$, 
this means that $e_\alpha$ itself admits a $\mathcal{C}^\infty$ extension to $\Omega$, i.e., $\alpha\in \mathcal{S}(\Omega)$, where 
$\mathcal{S}(\Omega)$ is as in \eqref{eq-somega}.
Therefore if $\alpha\not \in \mathcal{S}(\Omega)$, the corresponding term in the Laurent series of $T|_{\Omega\setminus Z}$ vanishes, 
and the series takes the form
\[ T|_{\Omega\setminus Z}=\sum_{\alpha\in \mathcal{S}(\Omega)} \wh{a}_\alpha(T|_{\Omega\setminus Z}) e_\alpha.\]

Now for each $\alpha \in \mathcal{S}(\Omega)$ define the map $a_\alpha:\OO(\Omega)\to \cx$ by $a_\alpha(T)=\wh{a}_\alpha(T|_{\Omega \setminus Z} ).$ Since both the restriction map and the coefficient functional $\wh{a}_\alpha$ are continuous, it follows that $a_\alpha$ is 
continuous. The extension of $e_\alpha$ from $\Omega\setminus Z$ to $\Omega$ is still the monomial $e_\alpha$, 
so the Fourier series of $T$ in $\OO(\Omega)$ is of the form \eqref{eq-nonannular}.
\end{proof}
Notice that each term of \eqref{eq-nonannular} 
is in $\mathcal{C}^\infty(\Omega)$ and by Proposition~\ref{prop-annular}, 
the series converges absolutely in $\mathcal{C}^\infty(\Omega \setminus Z)$ 
when  its terms are restricted to $\Omega\setminus Z$, i.e., it converges   uniformly along with all derivatives on those compact sets  in $\Omega$ which  do not intersect
$Z$.
 
\subsection{Extension to the relative completion}
Given a Reinhardt domain $D\subset \cx^n$, its \emph{relative completion} is the smallest 
relatively complete domain containing $D$ (see above before the statement of Theorem~\ref{thm-distn} for the definition of relative completeness of a domain.)  Notice that the relative completion of $D$ coincides with the unions of the sets $D^{(j)}$ of \eqref{eq-relcomp}, where the union is taken over those $j$ for which $D\cap \{z_j=0\}$ is nonempty. The following
general proposition, which encompasses classical examples of 
the Hartogs phenomenon, e.g. in the ``Hartogs figure", will be needed
to complete the proof of Theorem~\ref{thm-distn}.
\begin{prop}\label{prop-relcompletion}
        Let $D$ be a Reinhardt domain. Then each holomorphic function on $D$ extends holomorphically to the relative completion of $D$.
\end{prop}
\begin{proof} We may assume that $n\geq 2$ since each Reinhardt 
domain in the plane is automatically relatively complete. 
 If for each  $j$, the  intersection $D\cap \{z_j=0\}=\emptyset$, then the domain 
    $D$ is annular, and its relative completion is itself so there 
    is nothing to prove.
    Suppose  therefore that there is $1\leq j\leq n$ such that $D\cap \{z_j=0\}\not=\emptyset$. We need to prove that each function in 
    $\OO(D)$ extends holomorphically to  $D^{(j)}$.

    Without loss of generality 
    we can assume that $j=1$. Write the coordinates of a point $z\in \cx^n$ as $z=(z_1,\wt{z})$ where $\wt{z}\in \cx^{n-1}$. Let $f\in \OO(D)$. By Proposition~\ref{prop-nonannular} $f$ admits a Laurent series 
    representation
    \[ f= \sum_{\alpha\in \mathcal{S}(D)} a_\alpha(f)e_\alpha\]
    with $S(D)$ as in \eqref{eq-somega}, and the series converges 
    absolutely in $\mathcal{C}^\infty(D\setminus Z)$. To prove the proposition
    it suffices to show that the series in fact converges absolutely
    in $\mathcal{C}^\infty(D^{(1)})$,  which by Lemma~\ref{lem-weierstrass} is equivalent to the following:  for each point 
    $p\in D^{(1)}$ there is a compact neighborhood $M$ of $p$ in $D^{(1)}$ such that 
    \begin{equation}\label{eq-pm}
       \sum_{\alpha\in \mathcal{S}(D)} p_M(a_\alpha(f)e_\alpha) <\infty. 
    \end{equation} 
    
    We claim the following: for each $p\in D^{(1)}$ there is 
    a compact neighborhood $M$ of $p$ and a compact subset $K\subset D$ such that $p_M(e_\alpha)\leq p_K(e_\alpha)$ for 
    each $\alpha\in \mathcal{S}(D)$. Assuming the claim for a moment, we see that we have
    \[ p_M(a_\alpha(f)e_\alpha)= \abs{a_\alpha(f)}p_M(e_\alpha)\leq \abs{a_\alpha(f)}p_K(e_\alpha)= p_K(a_\alpha(f)e_\alpha),\]
    so that \eqref{eq-pm} follows since by \eqref{eq-pkconv} we do have $\sum_{\alpha\in \mathcal{S}(D)} p_{K}(a_\alpha(f)e_\alpha) <\infty$
    for a compact subset $K$ of $D$. Since each point $p\in D^{(1)}$ has such a neighborhood this completes the proof, modulo the claim above.
    
    To establish the claim, we may assume that $p\not \in D$, since 
    otherwise there is nothing to prove. Therefore $p\in D^{(1)}\setminus D$, and consequently, there is a $z\in D$ and 
    a $\lambda\in \D$ such that $p=(\lambda z_1, \wt{z})$, where $\wt{z}=(z_2,\dots, z_n)$. Let $K$ be a compact neighborhood 
    of $z$ in $D$ of the form $K=K_1\times \wt{K}$, where $\wt{K}\subset \cx^{n-1}$ and $K_1=\{\zeta\in \cx: \abs{\zeta-z_1}\leq \epsilon\}$ is a closed disk. Let $L_1$ be the disk $\{\zeta\in \cx: \abs{\zeta}\leq \abs{z_1}+\epsilon\}$,
    so that $K_1\subset L_1$, and
    \[ \gamma = {z_1}+ \frac{z_1}{\abs{z_1}}\epsilon \]
    is a point of maximum modulus (i.e. maximum distance from the origin) in both sets $K_1$ and $L_1$. Note that 
    \[ \abs{\gamma}= \abs{\frac{z_1}{\abs{z_1}}\cdot \abs{z_1}+ \frac{z_1}{\abs{z_1}}\epsilon}= \abs{z_1}+\epsilon.\]
    We set $M=L_1\times \wt{K}$. We will show that these sets $K, M$
    satisfy the conditions of the claim.

    Now let $\alpha\in \mathcal{S}(D)$. Since $D\cap \{z_1=0\}\not=\emptyset$, it follows that $\alpha_1\geq 0$. 
    Let $\wt{\alpha}=(\alpha_2,\dots, \alpha_n)\in \Z^{n-1}$, and 
    set
    \[B=  \sup_{\wt{w} \in \wt{K}} \abs{\wt{w}^{\wt{\alpha}}} \quad \text{where $\wt{w}=(w_2,\dots, w_n) \in \cx^{n-1}$ and $ \abs{\wt{w}^{\wt{\alpha}}}= \abs{w_2^{\alpha_2}} \cdots \abs{w_n^{\alpha_n}}$},\]
     so that we have
    \[ p_M(e_\alpha)=\sup_{w\in M}\abs{e_\alpha(w)}= \sup_{w_1\in L_1}\abs{w_1}^{\alpha_1}\cdot B = (\abs{z_1}+\epsilon)^{\alpha_1} \cdot B.\]
    On the other hand
    \[ p_K(e_\alpha) =\sup_{w\in K}\abs{e_\alpha(w)}=\sup_{w_1\in K_1}\abs{w_1}^{\alpha_1} \cdot B= (\abs{z_1}+\epsilon)^{\alpha_1} \cdot B.\]
    Consequently, in fact we have $p_M(e_\alpha)=p_K(e_\alpha)$, and the claim is proved, thus completing the proof. 
    \end{proof}

\subsection{End of proof of Theorem~\ref{thm-distn}}\label{sec-hartogs}
It only remains to put together the various pieces to note that all parts of Theorem~\ref{thm-distn} have been established. If $\Omega$ is annular, i.e. $\Omega$ has empty intersection with the set $Z$ of \eqref{eq-Z}, then Proposition~\ref{prop-annular} takes care of the 
complete proof. When $\Omega$ is allowed to be non-annular, we see from Proposition~\ref{prop-nonannular} that 
the Laurent series representation has only monomials which are smooth functions on $\Omega$. Now it is not difficult to see that 
the smallest log-convex relatively complete Reinhardt domain $\wh{\Omega}$ 
containing $\Omega$ can be constructed from $\Omega$ in two steps.
First, we construct the log-convex hull $\Omega_1$ of the set $\Omega\setminus Z$. Notice that $\Omega\setminus Z$  and $\Omega_1$ are both annular.
The second step consists of constructing the relative completion of the domain $\Omega_1$, thus obtaining  the domain $\wh{\Omega}$.
Now by Proposition~\ref{prop-logconvex}, the Laurent series of a holomorphic distribution on $\Omega$ converges absolutely in $\mathcal{C}^\infty(\Omega_1)$. Applying Proposition~\ref{prop-relcompletion} (with $D=\Omega_1$), we see that the 
Laurent series actually converges absolutely in the space $\mathcal{C}^\infty(\wh{\Omega})$. The sum of this series is the required
holomorphic extension of a given holomorphic distribution on $\Omega$. The result has been completely established.

\section{Missing monomials}\label{sec-missing}
In Theorem~\ref{thm-distn}, we considered the natural representation of the torus $\T^n$ on the space $\OO(\Omega)$ of holomorphic functions on a Reinhardt domain $\Omega$. In applications,
one often deals with a subspace of functions $Y\subset \OO(\Omega)$ such that
\begin{enumerate}
	\item the subspace $Y$ is invariant under the natural representation $\sigma$, i.e., if $f\in Y$ then $f\circ \rot_\lambda\in Y$ for each $\lambda\in 
	\T^n$, where $\rot_\lambda$ is as in \eqref{eq-rot},
		\item there is a locally convex topology on $Y$ in which it is complete, 
	and such that the inclusion map
	\[ j:Y\hookrightarrow\OO(\Omega)\]
is continuous, 
	and 
	\item when $Y$ is given this topology, the representation $\sigma$ restricts to a continuous representation on $Y$.
\end{enumerate}
The \emph{locus classicus} here is the theory of Hardy spaces on the disc. We can make the following elementary observation:
\begin{prop}\label{prop-missing}
	Let $Y$ and $\sigma$ be as above, and set 
	\[ \mathcal{S}(Y)= \{\alpha\in \Z^n: e_\alpha\in Y\}.\]
	Then the Laurent series of a function $f\in Y$ is of the form
	\[f = \sum_{\alpha\in \mathcal{S}(Y) }a_\alpha(f)e_\alpha.\]
	\end{prop}
\begin{proof}
	It suffices to show that if for $\alpha\in \Z^n$, if  the monomial $e_\alpha$ does not belong to $Y$,  we have that $a_\alpha(f)=0$, where $a_\alpha(f)$ is the Laurent coefficient of $f$ as in \eqref{eq-laurent}. 
	By part (1) of Proposition~\ref{prop-fourierproperties}, we see that $\ppi^\sigma_\alpha(f)\in Y$. By part (2) of the same proposition, taking $X$ to be the space $\OO(\Omega)$,
	 for $f\in Y$ we have that $\ppi^\sigma_\alpha(j(f))=j(\ppi^\sigma_\alpha(f))$, and from the description of the Fourier components of a holomorphic function in the proof of Theorem~\ref{thm-distn}, we see that 
	$\ppi^\sigma_\alpha(j(f))=a_\alpha(j(f))e_\alpha= a_\alpha(f)e_\alpha$. Therefore $j(\ppi^\sigma_\alpha(f))=a_\alpha(f)e_\alpha$, which contradicts the fact that $\ppi^\sigma_\alpha(f)\in Y$ unless $a_\alpha(f)=0$.
	\end{proof}
This simple observation can be called  the ``principle of missing monomials", since it says that certain monomials cannot occur in the Laurent series of the function $f$.  It can be thought to be the  reason behind several  phenomena associated to holomorphic functions. We consider two examples:

\begin{enumerate}[wide]
\item \textbf{Bergman spaces in Reinhardt domains:}
Let $\Omega$ be a Reinhardt domain  in $\cx^n$ and let $\lambda>0$ be a radial weight on $\Omega$,  i.e., for $z\in \Omega$ we have
\[ \lambda(z_1,\dots, z_n)= \lambda(\abs{z_1},\dots,\abs{z_n}).\]
The\emph{ $L^p$-Bergman space} $A^p(\Omega ,\lambda)$ is defined to be the subspace of the weighted $L^p$-space $L^p(\Omega, \lambda)$ consisting of holomorphic functions, where 
the norm on  the weighted $L^p$-space is given by
\[ \norm{f}_{L^p(\Omega,\lambda)}^p = \int_\Omega \abs{f}^p \lambda dV,\]
where $dV$ is the Lebesgue measure. It is well-known that $A^p(\Omega,\lambda)$ is a closed subspace of the Banach space $L^p(\Omega, \lambda)$ and therefore a Banach space (\cite{durenbergman}). It is also easy to see (using standard facts about $L^p$-spaces) that the natural representation $\sigma$ of $\T^n$ on $L^p(\Omega, \lambda)$ is continuous for 
$1\leq p <\infty$, so it follows that the representation on $A^p(\Omega, \lambda)$ is also continuous. It now follows from Proposition~\ref{prop-missing} that the Laurent series expansion of a function $f\in A^p(\Omega, \lambda)$ consists only of terms with monomials 
$e_\alpha$ such that $e_\alpha\in L^p(\Omega, \lambda)$. The case $\lambda \equiv 1$ of this fact was deduced by a different argument in \cite{ChEdMc}.

\item \textbf{Extension of holomorphic functions smooth up to the boundary:} Let $\Omega\subset \cx^n$ be a  Reinhardt domain such that the origin (which is the center of symmetry) is
on the boundary of $\Omega$. A classic example of this is the \emph{Hartogs triangle} $\{\abs{z_1}<\abs{z_2}<1\}$ in $\cx^2$.  In \cite{sibony}, the following extension theorem was proved: there is 
a complete Reinhardt  domain $V\subset \cx^n$ such that $\Omega\subset V$  and each function in the space $\mathcal{O}(\Omega)\cap \mathcal{C}^\infty(\ol{\Omega})$  of holomorphic functions on
$\Omega$ smooth up to the boundary extends holomorphically to the domain $V$. This was noted for the Hartogs triangle in \cite{sibony1975}, where $V$ is the unit bidisk.

To deduce this from the principle of missing monomials, it suffices to consider the case when $\Omega$ is bounded. 
 We notice that  $\mathcal{O}(\Omega)\cap \mathcal{C}^\infty(\ol{\Omega})$  is a  Fréchet space in its usual topology of 
uniform convergence on $\ol{\Omega}$ with all partial derivatives. A generating family of seminorms for this topology is 
given by the norms $\{p_k\}$ where $p_k(f)=\norm{f}_{C^k(\ol{\Omega})}$. The natural representation $\sigma$ on $\OO(\Omega)$ restricts to a continuous representation
 on $\mathcal{O}(\Omega)\cap \mathcal{C}^\infty(\ol{\Omega})$. Therefore, the principle of missing monomials applies and the only monomials $e_\alpha$ that occur in the Laurent 
 expansion of a function $f\in \mathcal{O}(\Omega)\cap \mathcal{C}^\infty(\ol{\Omega})$  are such that $e_\alpha \in\mathcal{O}(\Omega)\cap \mathcal{C}^\infty(\ol{\Omega})$. For such a multi-index $\alpha$, write the multiindex $\alpha= \beta-\gamma$, where $\beta_j=\max(\alpha_j,0)$ and  $\gamma_j=(-\alpha_j,0)$. Then $\beta, \gamma\in \N^n$ , and we can apply the differential operator $\left(\frac{\partial}{\partial z}\right)^\beta$ to obtain 
 \[ \left(\frac{\partial}{\partial z}\right)^\beta e_\alpha(z)= \left(\frac{\partial}{\partial z}\right)^\beta\frac{e_\beta(z)}{e_\alpha(z)}= \frac{\beta!}{z^\gamma}. \]
 Since $e_\alpha \in \mathcal{C}^\infty(\ol{\Omega})$ this means that $e_\gamma \in \mathcal{C}^\infty(\ol{\Omega})$, which is possible only if $\gamma=0$. Thus $\alpha\in \N^n$, and the Laurent series of  each function $f\in \mathcal{O}(\Omega)\cap \mathcal{C}^\infty(\ol{\Omega})$ is a Taylor series which converges in some complete (log-convex) Reinhardt domain $V$, and this $V$ must strictly contain $\Omega$, since $\Omega$ is not complete.

\end{enumerate}
		
	\section{Classical characterizations of holomorphic functions}\label{sec-morera}

In this section we show how one can avoid the machinery of generalized functions and weak 
derivatives altogether, and still use Fourier methods to prove the basic facts of function theory. We confine ourselves to one complex variable and the simple geometry of the disk.

\subsection{Goursat's characterization} 
The starting point of a traditional account of holomorphic functions of a single variable is 
typically  Goursat's definition ( \cite{goursat}): a function $f:\Omega\to \cx$ on an 
open subset $\Omega\subset \cx$ is holomorphic, if it is \emph{complex differentiable}, i.e., for each point $w\in \Omega$,
the limit 
\begin{equation}
	\label{eq-compder}{\lim_{z\to w}\frac{f(z)-f(w)}{z-w}}
\end{equation}
exists. The result that a holomorphic function in this sense is infinitely many times complex differentiable and even admits a convergent power series representation 
near each point is rightly celebrated as one of the most elegant and surprising in all of mathematics. Unfortunately, we cannot use it as a definition, 
if we want to apply the theory of abstract Fourier expansions as developed in Section~\ref{sec-abstract}.  Denoting by $G(\Omega)$ the collection of holomorphic functions in the sense 
of Goursat in an open set $\Omega\subset \cx$, we notice that the space $G(\Omega)$ does not have 
a nice a priori linear locally convex topology  in which it is complete and such that when $\Omega$ is a disc or an annulus, the natural action of the group $\T$ on the space $G(\Omega)$ is a continuous representation.  Though Goursat's definition
carries the weight of a century of academic tradition, we will start from an alternative definition  which lends itself better to the application of the methods of Section~\ref{sec-abstract}.
We also note that  the characterization of holomorphic functions by complex-differentiability cannot be used for natural generalizations of complex analysis, e.g. quaternionic analysis,   analysis  on  Clifford algebras etc. (see \cite[pp. 87--93]{gilbertmurray}).

\subsection{Morera's definition} 
Let $\Omega$ be an open subset of the complex plane $\cx$, and let $f:\Omega\to \cx$
be a continuous function.  In honor of \cite{morera}, let us say that the function 
$f$ is \emph{holomorphic in the sense of Morera} (\emph{Morera-holomorphic} for 
short) if for each triangle
$T$ contained (with its interior) in $\Omega$, we have the vanishing of the  complex 
line integral of $f$ around the boundary of $T$:
\begin{equation}
	\label{eq-cauchygoursat}\int_{\partial T} f(z)dz=0,
\end{equation}
where $\partial T$ denotes the boundary of $T$, oriented counterclockwise.
For an open set $\Omega\subset\cx$, let us denote by $\om(\Omega)$ the collection of Morera-holomorphic functions on $\Omega$.  It is known by Morera's theorem that a Morera-holomorphic function is Goursat-holomorphic, and one can develop function theory 
starting from Morera's definition (see \cite{heffter, mac-wil}). Notice that the 
a priori regularity of Morera-holomorphic functions (assumed to be only continuous)
 is even less than that assumed for Goursat-holomorphic functions (assumed also to admit the 
 limit \eqref{eq-compder} at each $w$.)

It is immediate from the definition that $\om(\Omega)$ is a closed linear subspace 
of the Fréchet space $\mathcal{C}(\Omega)$ of continuous functions. ``Closed" means that 
the limit of a sequence of Morera-holomorphic functions converging uniformly on 
compact subsets of $\Omega$ is itself Morera-holomorphic, a fact that was already noted 
in \cite{morera1}.  The proof of this crucial fact  starting from the Goursat definition
must pass through a lengthy  development of integral representations, 
so this is definitely a pedagogical advantage of Morera's definition over Goursat's.

The notion of Morera-holomorphicity is local: i.e., $f\in \om(\Omega)$ if and only if 
there is an open cover $\{\Omega_j\}_{j\in J}$ of $\Omega$ such that $ f|_{\Omega_j}\in \om(\Omega_j)$ for each $j$.  One half of this claim is trivial, and for the other half, 
for a triangle $T$ in $\Omega$, we can 
perform repeated barycentric subdivisions till the triangles so formed are each contained in some element of the open cover $\{\Omega_j\}_{j\in J}$. We therefore conclude that 
$\om$ is a \emph{sheaf} of Fr\'{e}chet spaces on $\cx$.

The following local description of Morera-holomorphic functions is  well-known, and a proof can be found in e.g. \cite[pp. 186-189]{remmert}. 
\begin{prop}\label{prop-equivalent}
	Let $\Omega\subset\cx$ be convex. Then the following statements about a continuous function  $f\in \mathcal{C}(\Omega)$ are equivalent:
	\begin{enumerate}[label=(\Alph*)]
		\item $f\in \om(\Omega)$.
		\item $f$ has a holomorphic primitive, i.e., there is an $F$ which is complex-differentiable on $\Omega$ and $F'=f$.
		\item for each piecewise $\mathcal{C}^1$ closed path $\gamma$ in $\Omega$
		we have 
		\begin{equation}
			\label{eq-cauchyc1}\int_\gamma f(z)dz=0.
		\end{equation}
	\end{enumerate}
\end{prop}
Recall that assuming (A), the primitive $F$ in (B) 
is constructed by fixing $z_0\in \Omega$, and setting
$ \displaystyle{F(z)= \int_{[z_0,z]}f(\zeta)d\zeta}$
where $[z_0,z]$ denotes the line segment from $z_0$ to $z$.
Proposition~\ref{prop-equivalent}  allows us to give examples of Morera-holomorphic functions. Recall from \eqref{eq-ealpha} that for an integer $n$, we use the notation $	e_n(z)=z^n$
for the holomorphic monomials.
\begin{prop} \label{prop-en} If $n\geq 0$ then
	$e_n\in \om(\cx)$ and if $n<0$ then  $e_n\in \om(\cx\setminus\{0\})$.
\end{prop}
\begin{proof} First note that $e_n$ is continuous, on all of $\cx$ if $n\geq0$ and on $\cx\setminus\{0\}$ if
	$n<0$.
	If $g_n(z)= \frac{z^n}{n+1}$ for $n\not=-1$, we can verify from 
	the definition \eqref{eq-compder} that $g_n$ is complex-differentiable and $g_n'=e_n$ so that by part (B) of
	Proposition~\ref{prop-equivalent} the result follows for $n\not=-1$. For $n=-1$, we can construct for each $p\in \cx\setminus\{0\}$, a local primitive of $e_{-1}$ near $p$ by setting
	\[ g_{-1}(z)=\ln\abs{z}+ i \arg z,\]
	where $\arg$ denotes a branch of the argument defined near the point $p$. A direct computation shows that $g_{-1}'=e_{-1}$
	near $p$, so that again we see that $e_{-1}\in \om(\cx\setminus\{0\}).$\end{proof}
\subsection{Products of Morera-holomorphic functions} In the proof of Theorem~\ref{thm-distn}, an important role is played by the fact  that if $U$ is a holomorphic
distribution  (in the sense of \eqref{eq-holdistn})and $f$ is a holomorphic function (i.e. a holomorphic distribution which  is $\mathcal{C}^\infty$, Section~\ref{sec-holfn} above), then the product distribution $fU$  is also a holomorphic distribution. This is an immediate 
consequence of the distributional Lebniz formula \eqref{eq-leibniz}. A similar result,
proved in \cite{mac-wil},
will be needed in order to develop the properties of holomorphic functions
starting from Morera's definition. 
\begin{prop}
	\label{prop-gengoursat}
	Let $f,g\in \om(\Omega)$, and assume that $g$ is locally Lipschitz at each point, i.e. for each $w\in \Omega$ and 
	each compact $K\subset \Omega$ such that $w\in K$, there is an
	$M>0$ such that for $z\in K$ we have
	\begin{equation}\label{eq-mw}
		\abs{g(z)-g(w)}\leq M\abs{z-w}.  
	\end{equation}
	Then the product $fg$ also belongs to $\om(\Omega)$.
\end{prop}
The proof is based on a version of  the classical Goursat lemma (\cite{goursat, pringsheim}).
This is of course  the
main ingredient in the standard textbook proof of the Cauchy theorem for triangles for 
Goursat-holomorphic functions.  Recall that two triangles are \emph{similar} if they have 
the same angles. 
\begin{lem}\label{lem-goursat}
	Let $\Omega$ be an open subset of $\cx$ and let $\lambda$ be a complex valued function
	defined on the set of triangles contained in $\Omega$ such that  the following two conditions are satisfied:
	\begin{enumerate}
		\item $\lambda$ is additive in the following sense: if a triangle $\Delta$ is represented as 
		a union of smaller triangles $\Delta=\bigcup_{k=1}^n \Delta_k$  with pairwise disjoint interiors
		then 
		\begin{equation}
			\label{eq-additive}
			\lambda(\Delta)= \sum_{k=1}^n \lambda(\Delta_k).
		\end{equation}
		\item  For each $w\in \Omega$ and each triangle $\Delta_0$, we have
		\begin{equation}\label{eq-areolar1}
		\lim_{\substack{\Delta\downarrow w\\ \Delta\sim \Delta_0 }} \frac{\lambda(\Delta)}{\abs{\Delta}}=0,
		\end{equation}
		where $\abs{\Delta}$ denotes the area of the triangle $\Delta$, and the limit is taken along the
		family of triangles  similar to $\Delta_0$ and containing the point $w$, as these
		triangles shrink to the point $w$. 
	\end{enumerate}
Then $\lambda \equiv 0$. 
\end{lem}

\begin{proof}	For completeness, we recall the classic argument. Let $T$ be a triangle 
	contained in $\Omega$. We construct a sequence of triangles $\{T_k\}_{k=0}^\infty$  with $T_0=T$ using the following recursive procedure. 	Assuming that $T_k$ has been constructed, we divide $T_k$ into four similar triangles with half the diameter of 
	$T_k$ by three line segments each parallel to a side of $T_k$ and passing through the midpoints of the other two sides.  Denote the four triangles so obtained by $\Delta_j, 1\leq j \leq 4$. Then, by \eqref{eq-additive}, we have
	\[\lambda(T_k)=  \sum_{j=1}^4 \lambda(\Delta_j).\]
	Choose $T_{k+1}$ to be one of $\Delta_j, 1\leq j \leq 4$ such that the value of $\abs{\lambda(T_{k+1})}$ is 
	the largest.  Then, by the triangle inequality we have 
	$ \abs{\lambda(T_k)}\leq 4\abs{\lambda(T_{k+1})}, $ and by induction it follows that 
	\begin{equation}
		\label{eq-lambdat}
	\abs{\lambda(T)}\leq 4^k\abs{\lambda(T_{k})}= 4^k \abs{T_k}\cdot \frac{\abs{\lambda(T_{k})}}{\abs{T_k}}= \abs{T} \cdot\frac{\abs{\lambda(T_{k})}}{\abs{T_k}}, \end{equation}
	where in the last step we have used the fact that $T_{k+1}$ has one-fourth the area of  $T_k$, so $\abs{T_k}=4^{-k}\abs{T_0}= 4^{-k}\abs{T}$. Since the diameters of the $T_k$ 
	go to zero, by compactness,  there is a unique point $w$ in the intersection $\bigcap_{k=0}^\infty T_k$. Since the family $\{T_k\}$ is a subfamily of all the triangles containing $w$, and each $T_k$ is similar to $T_0$, therefore by letting $k\to \infty$ in \eqref{eq-lambdat} and using \eqref{eq-areolar1}
	the result follows. 
  \end{proof}

\begin{proof}[Proof of Proposition~\ref{prop-gengoursat}]
	For a triangle $\Delta\subset \Omega$, define 
	\[ \lambda(\Delta)= \int_{\partial \Delta} f(z)g(z) dz.\]
	To prove the result, we need to show that $\lambda\equiv0$.  Since condition \eqref{eq-additive} of Lemma~\ref{lem-goursat} is obvious, we need to show  \eqref{eq-areolar1} to complete the proof. Let $w\in \Omega$,   let $K$ be a compact neighborhood of $w$ in $\Omega$, and denote by $M$ the Lipschitz constant corresponding
	to this $w$ and this $K$ in \eqref{eq-mw}. Now,
	let $\Delta_0$ be a triangle 
	and let $\Delta$ be a triangle similar to $\Delta_0$ such that $w\in \Delta \subset K$. Then 
	observe that, by the hypothesis of Morera-holomorphicity of $f$ and $g$ we have
	\[ \lambda(\Delta)= \int_{\partial\Delta} \left(f(z)-f(w)\right)\left(g(z)-g(w)\right)dz.\]
	Therefore, denoting the perimeter of the triangle $\Delta$ by $\abs{\partial\Delta}$,
	\begin{align*}
\abs{\lambda(\Delta)}&\leq \sup_{z\in \partial \Delta}\abs{\left(f(z)-f(w)\right)\left(g(z)-g(w)\right)} \abs{\partial \Delta}\\
& \leq M \cdot  \abs{\partial \Delta}\cdot\sup_{z\in \partial \Delta}\left(\abs{f(z)-f(w)} \abs{z-w}\right)\\
&\leq  M \cdot  \abs{\partial \Delta}\cdot \mathrm{diam}(\Delta)\cdot\sup_{z\in \partial \Delta}\abs{f(z)-f(w)} \\
&= M\cdot\abs{\Delta} \cdot \left( \frac{ \abs{\partial \Delta}\cdot \mathrm{diam}(\Delta)}{\abs{\Delta}}\right)\cdot\sup_{z\in \partial \Delta}\abs{f(z)-f(w)} \\
&= M\cdot\abs{\Delta} \cdot \left( \frac{ \abs{\partial \Delta_0}\cdot \mathrm{diam}(\Delta_0)}{\abs{\Delta_0}}\right)\cdot\sup_{z\in \partial \Delta}\abs{f(z)-f(w)},
	\end{align*}
	where in the last step we use the fact that $\Delta$ and $\Delta_0$ are similar, so the quantity in parentheses (which is clearly invariant under dilations) is the same. So we have
	
\[ \frac{\abs{\lambda(\Delta)}}{\abs{\Delta}}\leq C\cdot\sup_{z\in \partial \Delta}\abs{f(z)-f(w)},\]
for a constant $C$ independent of the triangle $\Delta$ as long as $\Delta$ is similar to $\Delta_0$ and $w\in \Delta \subset K$. Letting $\Delta$ shrink to $w$, we have \eqref{eq-areolar1} and the proof is complete.
\end{proof}

\subsection{Fourier expansion of a Morera-holomorphic function}
We will now prove the following analog of Theorem~\ref{thm-distn}. In particular, it shows that holomorphic functions in the sense of Morera are identical to the holomorphic distributions considered in Section~\ref{sec-distn}.
\begin{thm}\label{thm-morera}
	Let $f\in \om(\D)$ where $\D=\{\abs{z}<1\}$ is the open unit disc. Then there is a sequence $\{a_n\}_{n=0}^\infty$ 
	of complex numbers, such that 
	\begin{equation}
		\label{eq-taylor2}f= \sum_{n=0}^\infty a_n e_n,
	\end{equation}
	where $e_n$ is as in \eqref{eq-ealpha}, and  the series on the right
	converges absolutely in   $\mathcal{C}(\D)$ to the function $f$.
\end{thm}
	Let $\sigma$ be  the natural representation of $\T$ on $\mathcal{C}(\D)$ given by 
	\begin{equation}\label{eq-sigmacd}
		 \sigma_\lambda(f)(z)=f(\lambda z), \quad \lambda\in \T, z\in \D.
	\end{equation}

\begin{prop}\label{prop-rep}
	The space $\om(\D)$ is invariant under $\sigma$ and the resulting representation of
	$\T$ on $\om(\D)$ is continuous.
\end{prop}
\begin{proof}
	Let $\lambda\in \T$, let $T$ be a triangle in $\D$. Notice that
	\[ \lambda^{-1}T=\{ \lambda^{-1}z:z\in T\}\]
	is itself a triangle, 	and we have 
	\[ \int_{\partial T}(\sigma_\lambda f)(z)dz =\int_{\partial T} f(\lambda z)dz =\lambda^{-1}\int_{\partial (\lambda^{-1}T)}f(w)dw=0.\]
	It follows that $\om(\D)$ is invariant under $\sigma$.

	It suffices to show that the representation $\sigma$ is continuous on $\mathcal{C}(\D)$.
	For $0<r<1$, let $p_r$ be the seminorm on $\mathcal{C}(\D)$ given by 
		\begin{equation}\label{eq-pr}
		p_r(f)= \sup_{\abs{z}\leq r} \abs{f(z)}.
	\end{equation}
	It is clear that $p_r(\sigma_\lambda(f))= p_r(f)$ for each $0<r<1, \lambda\in\T $ and
	 $f\in \mathcal{C}(\D)$. Also $f_j\to f$ in $\mathcal{C}(\D)$ if and only if 	for each $r$, we have $p_r(f-f_j)\to 0$, so the family $\{p_r\}$ is a $\sigma$-invariant family of seminorms that 
	 generates the topology of $\mathcal{C}(\D)$. Further, given $f\in \mathcal{C}(\D)$, by uniform continuity,
	  for each $0<r<1$, we have
	  \[ p_r(\sigma_\lambda(f)-f)= \sup_{\abs{z}\leq r} \abs{f(\lambda z)-f(z)}\to 0\quad \text{ as } \lambda\to 1,\]
	  so that $\lim_{\lambda\to 1} \sigma_\lambda(f)=f$ in the space $\mathcal{C}(\D)$. Therefore 
	  both conditions of Proposition~\ref{prop-continuity} are satisfied, and the representation 
	 $\sigma$ is continuous.  
\end{proof}

In view of the above, the machinery of Section~\ref{sec-abstract} applies. We now compute the 
Fourier components \eqref{eq-fouriermode}.
\begin{prop}
	\label{prop-fourier} For $f\in \om(\D)$, and with $\sigma$ the natural representation \eqref{eq-sigmacd}, the Fourier components of $f$ are of the form:
	\[ \ppi^\sigma_n(f)= \begin{cases}a_ne_n&\text{ if } n\geq 0\\
		0 & \text{ if } n<0,
	\end{cases}
\]
	where $a_n\in \cx$ and $e_n(z)=z^n$.
\end{prop}
The proof will use the following lemma:
\begin{lem}\label{lem-radial2} A radial function in $\om(\D\setminus \{0\})$ is constant. 
\end{lem} 
\begin{proof} Let  $f\in \om(\D\setminus \{0\})$ be radial, and 
	define the  complex valued continuous function 
	$u$ on the interval $(0,1)$ by restriction, $u(r)=f(r)$, so that we have
	$f(z)=u(\abs{z})$ by the radiality of $f$. To prove the theorem, it suffices to show that  $u$ is a constant  .

	Fix $0\leq \alpha<\beta <\pi$ and $0<\rho<1$. For $R$ in the interval $(\rho,1)$ consider the 
	curvilinear quadrilateral defined by 
	\begin{equation}
		\label{eq-S}S(R)=\{re^{i\theta}: \rho\leq r \leq R, \alpha\leq \theta \leq \beta\},
	\end{equation}
	and notice that $S(R)$ lies in the upper half disc, which is convex.
	The region $S(R)$ is  bounded by the two circular arcs
	\[ AB= \{\rho e^{i\theta} : \alpha \leq \theta\leq \beta \},\quad
	CD= \{Re^{i\theta} : \alpha \leq \theta\leq \beta \},\]
	along with the two radial line segments
	\[ AD=\{re^{i\alpha}: \rho\leq r \leq R\},\quad BC=\{re^{i\beta}: \rho\leq r \leq R\}.\]
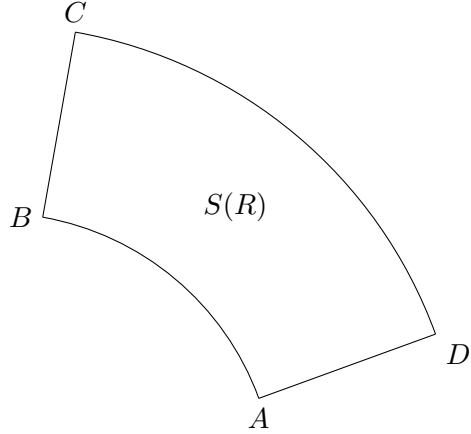
	\begin{center}
		\vspace{5mm}
		\begin{tikzpicture}[scale=1.25]
			\draw (2.81907,1.02606)-- (4.69846, 1.71010);
			\node [below] at (2.81907,1.02606) {$A$};
			\node [below right] at (4.69846, 1.71010) {$D$};
			\draw (0.52094, 2.95442)-- (0.86824, 4.92403);
			\node [above] at (0.86824, 4.92403) {$C$};
			\node [left] at (0.52094, 2.95442) {$B$};
			\draw (2.81907,1.02606) arc [radius=3, start angle=20, end angle=80];
			\draw (4.69846, 1.71010) arc [radius=5, start angle=20, end angle=80];
			\node at (2.57115,3.06417) {$S(R)$};
		\end{tikzpicture}
	\end{center}
	
	Orient the boundary $\partial S(R)$ counterclockwise. We can  write
	\[ \int_{\partial S(R)}f(z)dz =\int_{AD}+\int_{DC}+\int_{CB}+\int_{BA}=0, \]
	where the vanishing of the integral follows from part (c) of Proposition~\ref{prop-equivalent} above. 
	Parametrizing $AD$ by $z=re^{i\alpha}$ where $\rho \leq r \leq R$ , and using $dz=e^{i\alpha} dr$ 
	we get 
	\[ \int_{AD} f(z)dz =\int_{\rho}^R f(re^{i\alpha})e^{i\alpha}dr=e^{i\alpha}\int_{\rho}^R u(r)dr. \]
	Similarly, 
	\[ \int_{CB}f(z)dz=-e^{i\beta}\int_{\rho}^Ru(r)dr.\] 
	Now, parametrizing 
	$DC$ by $z=Re^{i\theta},  \alpha \leq \theta\leq \beta $, we have $dz=R ie^{i\theta}d\theta$, so 
	\[\int_{DC}f(z)dz =\int_{\theta=\alpha}^\beta f(R e^{i\theta})\cdot R i e^{i\theta}d\theta= Ru({R}) \int_{\alpha}^\beta ie^{i\theta}d\theta=Ru(R)(e^{i\beta}-e^{i\alpha}). \]
	Similarly,
	\[ \int_{BA} f(z)dz=\rho u(\rho)(e^{i\beta}-e^{i\alpha}).\]
	
	Therefore, adding the four integrals we have
	\[ \int_{\partial S}f(z)dz=\left({\rho}u({\rho})-Ru(R)+\int_{\rho}^R u(r)dr\right)(e^{i\alpha}-e^{i\beta})=0,\]
	so that we have the relation
	\begin{equation}
		\label{eq-id1}\int_{\rho}^R u(r)dr= Ru(R)- {\rho}u({\rho}),
	\end{equation}
	which can be written as
	\[u(R)= \frac{1}{R}\left(  \int_{\rho}^R u(r)dr+{\rho}u({\rho})\right). \]
	Since $u$ is continuous, the right hand side is a function of $R$ which is continuously 
	differentiable on $(\rho,1)$. Thus $u\in \mathcal{C}^1 (\rho,1)$.
	Differentiating both sides of \eqref{eq-id1} with respect to $R$
	we have $u(R)=Ru'(R)+u(R)$ so that 
	we have $u'(R)=0$ for $R\in (\rho,1)$. But $0<\rho<1$ can be chosen arbitrarily, so it follows that $u'\equiv 0$
	on $(0,1)$ and the result follows.
\end{proof}

\begin{proof}[Proof of Proposition~\ref{prop-fourier}]

	
	Let $n\in\Z$ and for simplicity of notation, let $f_n=\ppi^\sigma_n(f)$.
	Then by Proposition~\ref{prop-fourierproperties}, we see that $f_n$ lies in the Fourier mode
	$[\om(\D)]_n^\sigma$ so that $f_n\in \om(\D)$ and for $\lambda\in T, z\in \D$ we have
	$f_n(\lambda z)= \lambda^n f_n(z)$. Since clearly $f_n|_{\D\setminus\{0\}}\in \om(\D\setminus \{0\})$, it follows by Proposition~\ref{prop-gengoursat} that the function $h_n$ on 
	$\D\setminus\{0\}$ defined by
	\[ h_n=e_{-n}\cdot f_n\]
	lies in $\om(\D\setminus \{0\})$. Further for $\lambda\in T, z\in \D$ we have
	\[h_n(\lambda z)=e_n(\lambda z)f_n(\lambda z)= \lambda^{-n}e_{-n}(z)\cdot \lambda^n f_n(z)= h_n(z),\]
	so $h_n$ is radial, and hence by Lemma~\ref{lem-radial2}, $h_n$ is a constant, which we call $a_n$. Therefore, on $\D\setminus \{0\}$, we have $f_n=a_ne_n$, so the product $a_ne_n$ extends for each $n$ to a continuous function on $\D$. If $n<0$, this is possible only if $a_n=0$, and this completes the proof. 
\end{proof}

\subsection{Conclusion of the proof of Theorem~\ref{thm-morera}}
We will first show that the series $\sum_{n=0}^\infty a_ne_n$  converges absolutely in the 
space $\mathcal{C}(\D)$. It suffices to show that for each $0<r<1$, we have
\[ \sum_{n=0}^\infty p_r(a_ne_n)<\infty,\]
where $p_r$ is as in \eqref{eq-pr}. Fix an $r$ with $0<r<1$ and let $\rho$ be such that 
$r<\rho <1$. Applying the Cauchy estimate \eqref{eq-cauchy} to the seminorm $p_\rho$ we see that for each $n\in \Z$:
\[ p_\rho(\ppi^\sigma_n f)\leq p_\rho(f)\]
Thanks to Proposition~\ref{prop-fourier} we have for each $n\geq 0$,
\[ p_r(\ppi^\sigma_n f)= \sup_{\abs{z}\leq r} \abs{a_nz^n}= \abs{a_n}r^n \leq \left(\frac{r}{\rho}\right)^n \cdot\abs{a_n}\rho^n\leq \left(\frac{r}{\rho}\right)^n\cdot p_\rho(f).\]
Therefore we have
\[ \sum_{n=-\infty}^\infty{ p_r(\ppi^\sigma_n f)} = \sum_{n=0}^\infty{ p_r(\ppi^\sigma_n f)} \leq \left(\sum_{n=0}^\infty\left(\frac{r}{\rho}\right)^n \right)\cdot p_\rho(f)= \frac{1}{1-\frac{r}{\rho}} \cdot p_\rho(f)<\infty.\]
This proves that the series $\sum_{n=0}^{\infty} a_n e_n$ converges absolutely in $\mathcal{C}(\D)$. Let $g$ be its sum. Since the partial sums are all in $\om(\D)$ by Proposition~\ref{prop-en}, we see that $g \in \om(\D)$. However, it follows from Corollary \ref{cor-cesaro} that the sum of the series is $f$, therefore $f =g$, and the proof of the theorem is complete.

\subsection{Pompeiu's characterization of holomorphic functions}
Goursat's definition is non-quantitative, since it is framed in terms of the 
 \emph{existence} of the limit \eqref{eq-compder}, and  does not provide a  way to measure the degree of non-holomorphicity of a function. For example, if $\epsilon$ is  small and 
 nonzero,
all it says  about the functions $f(z)=z+\epsilon \ol{z}$ and $g(z)=\ol{z}+\epsilon {z}$ is that both are non-holomorphic.   Morera's definition of a holomorphic function does not have this shortcoming, as the integral $\int_{\partial T}f(z)dz$ gives a measure of 
the amount of non-holomorphicity of $f$ on the triangle $T$. 

It is possible to normalize and localize this measure of non-holomorphicity,
as was realized by Pompeiu (see \cite{pompeiu,mitrea}).  The quantity 
\begin{equation}
	\label{eq-average}\frac{1}{\abs{T}}\int_{\partial T} f(z)dz
\end{equation}
is a numerical measure of the ``average density  of non-holomorphicity" of a continuous function $f$ on a triangle (or other region with piecewise smooth boundary) $T$, where 
$\abs{T}$ denotes the area of  $T$.  To localize this, we can consider,
for a point $w\in \Omega$, and a continuous function $f:\Omega\to \cx$, the following limit 
(called  the ``areolar derivative" by Pompieu) as 
a measure of the degree of non-holomorphicity of $f$ at the point $w$:
\begin{equation}
	\label{eq-pompeiu}\lim_{T\downarrow w} \frac{1}{\abs{T}}\int_{\partial T} f(z)dz
\end{equation}
where the limit is taken over the family of triangles containing the point $w$ and contained in $\Omega$, as these
triangles shrink to the point $w$. Let us say that a function is \emph{Pompeiu-holomorphic} if the limit \eqref{eq-pompeiu}
exists at each $w\in \Omega$, and is equal to zero. The following simple and well-known observations 
clarify the meaning of this notion:
\begin{prop}\begin{enumerate}[wide, label=(\alph*)]
		\item A continuous function $f:\Omega\to \cx$ is Pompieu-holomorphic if and only if it is Morera-holomorphic.
		\item If the function $f$ is in $\mathcal{ C}^1(\Omega)$, then the limit \eqref{eq-pompeiu} exists, 
		and is equal to $\displaystyle{2i \frac{\partial f}{\partial \ol{z}}(w)}$.
	\end{enumerate}
	
\end{prop}
\begin{proof} For part (a), if  $ f$ is Morera-holomorphic, then the quantity \eqref{eq-average} vanishes for each triangle $T$, so
	the limit \eqref{eq-pompeiu} vanishes. Conversely, suppose that the limit \eqref{eq-pompeiu} vanishes at each point $w\in \Omega$, and let $T$ be a triangle 
in $\Omega$.  In Lemma~\ref{lem-goursat}, if we take
$\lambda(\Delta)=\int_{\partial \Delta}f(z)dz$, then the additivity condition \eqref{eq-additive} is clear, and the limit condition \eqref{eq-areolar1} holds by hypothesis. 
The result follows. 

For part (b), using Stoke's theorem
\begin{align*}\frac{1}{\abs{T}}\int_{\partial T} f(z)dz&=\frac{1}{\abs{T}}\int_{\partial T} f(z)dx+i f(z)dy= \frac{1}{\abs{T}}\int_T \left(i \frac{\partial f}{\partial x}(z)- 
	\frac{\partial f}{\partial y}(z)\right)dx\wedge dy\\&= 2i \int_T\frac{\partial f}{\partial \ol{z}}dx\wedge dy. \end{align*}

Since the last integral has a continuous integrand, we may take the limit as the triangle $T$ shrinks to the point $w$ to obtain:
\[ \lim_{T\downarrow w} \frac{1}{\abs{T}}\int_{\partial T} f(z)dz= 2i \frac{\partial f}{\partial \ol{z}}(w).\]
\end{proof}


\bibliographystyle{alpha}
\bibliography{power}
\end{document}